\documentclass[a4paper,english]{article}
\pdfoutput=1
\usepackage[T1]{fontenc}
\usepackage[utf8]{inputenc}
\usepackage[english]{babel}
\usepackage{amsthm}
\usepackage{amsmath}
\usepackage{amssymb}
\usepackage{graphicx}
\usepackage{esint}
\usepackage[unicode=true,pdfusetitle,
 bookmarks=true,bookmarksnumbered=false,bookmarksopen=false,
 breaklinks=false,pdfborder={0 0 0},backref=false,colorlinks=false]
 {hyperref}
\hypersetup{pdftitle={Polynomial and rational inequalities on analytic Jordan arcs and domains},
 pdfauthor={Sergei Kalmykov and Béla Nagy}}

\makeatletter

\usepackage{enumitem}		% customizable list environments
\theoremstyle{plain}
\newtheorem{thm}{\protect\theoremname}
  \theoremstyle{plain}
  \newtheorem*{thm*}{\protect\theoremname}
  \theoremstyle{plain}
  \newtheorem{prop}[thm]{\protect\propositionname}
  \theoremstyle{plain}
  \newtheorem{lem}[thm]{\protect\lemmaname}

\date{}

\makeatother

  \providecommand{\lemmaname}{Lemma}
  \providecommand{\propositionname}{Proposition}
  \providecommand{\theoremname}{Theorem}
\providecommand{\theoremname}{Theorem}

\begin{document}
\global\long\def\bC{\mathbf{C}}
\global\long\def\bCi{\mathbf{C}_{\infty}}
\global\long\def\bZ{\mathbf{Z}}
\global\long\def\bT{\mathbb{T}}
\global\long\def\bR{\mathbf{R}}
\global\long\def\res{\mathrm{Res}}
\global\long\def\bD{\mathbb{D}}
\global\long\def\sign{\mathrm{sgn}}
\global\long\def\dist{\mathrm{dist}}
\global\long\def\ord{\mathrm{ord}}
\global\long\def\inter{\mathrm{Int}}

\title{Polynomial and rational inequalities on analytic Jordan arcs and
domains}

\author{Sergei Kalmykov and Béla Nagy}

\maketitle
\begin{center}
{\large{}Dedicated to Professor Vilmos Totik on his sixtieth birthday}
\par\end{center}{\large \par}

\begin{abstract}
In this paper we prove an asymptotically sharp Bernstein-type inequality
for polynomials on analytic Jordan arcs. Also a general statement
on mapping of a domain bounded by finitely many Jordan curves onto
a complement to a system of the same number of arcs with rational
function is presented here. This fact, as well as, Borwein-Erdélyi
inequality for derivative of rational functions on the unit circle,
Gonchar-Grigorjan estimate of the norm of holomorphic part of meromorphic
functions and Totik's construction of fast decreasing polynomials
play key roles in the proof of the main result. 
\footnote{
This is author accepted manuscript, including a few typo corrections.
The published version of the paper is available at 
\href{http://dx.doi.org/10.1016/j.jmaa.2015.05.022}{DOI: 10.1016/j.jmaa.2015.05.022}.
}

Classification (MSC 2010): 41A17, 30C20, 30E10
\end{abstract}

\section*{Introduction}

Let $\bT:=\left\{ z\in\bC:|z|=1\right\} $ denote the unit circle,
$\bD:=\left\{ z\in\bC:|z|<1\right\} $ denote the unit disk and $\bCi:=\bC\cup\left\{ \infty\right\} $
denote the extended complex plane. We also use $\bD^{*}:=\left\{ z\in\bC:\ \left|z\right|>1\right\} \cup\left\{ \infty\right\} $
for the exterior of the unit disk and $\left\Vert .\right\Vert _{K}$
for the sup norm over the set $K$.

First, we recall a Bernstein-type inequality proved by Borwein and
Erdélyi in \cite{MR1433285} (and in a special case, by Li, Mohapatra
and Rodriguez in \cite{MR1332889}). We rephrase their inequality
using potential theory (namely, normal derivatives of Green's functions)
and for the necessary concepts, we refer to \cite{MR1485778} and
\cite{MR1334766}. Then we present one of our main tools, the ``open-up''
step in Proposition \ref{prop:open-up}, similar step was also discussed
by Widom, see \cite{MR0239059}, p. 205--206 and Lemma 11.1. This
way we switch from polynomials and Jordan arcs to rational functions
and Jordan curves. Then we use two conformal mappings, $\Phi_{1}$
and $\Phi_{2}$ to map the interior of the Jordan domain onto the
unit disk and to map the exterior of the domain onto the exterior
of the unit disk respectively. We transform our rational function
with $\Phi_{1}$ and ``construct'' a similar rational function (approximate
with another, suitable rational function) so that the Borwein-Erdélyi
inequality can be applied.

Our main theorem is the following.
\begin{thm}
\label{thm:main} Let $K$ be an analytic Jordan arc, $z_{0}\in K$
not an endpoint. Denote the two normals to $K$ at $z_{0}$ by $n_{1}\left(z_{0}\right)$
and $n_{2}\left(z_{0}\right)$. Then for any polynomial $P_{n}$ of
degree $n$ we have 
\begin{multline*}
\left|P_{n}'\left(z_{0}\right)\right|\le\left(1+o\left(1\right)\right)n\left\Vert P_{n}\right\Vert _{K}\\
\cdot\max\left(\frac{\partial}{\partial n_{1}\left(z_{0}\right)}g_{\bC_{\infty}\setminus K}\left(z_{0},\infty\right),\ \frac{\partial}{\partial n_{2}\left(z_{0}\right)}g_{\bC_{\infty}\setminus K}\left(z_{0},\infty\right)\right)
\end{multline*}
where $o\left(1\right)$ depends on $z_{0}$ and $K$ only and tends
to $0$ as $n\rightarrow\infty$.
\end{thm}
Remark. This theorem was formulated as a conjecture in \cite{MR3019778}
on page 225.

Theorem \ref{thm:main} is asymptotically sharp as the following theorem
shows.
\begin{thm}
\label{thm:sharpness} Let $K$ be a finite union of disjoint, $C^{2}$
smooth Jordan arcs and $z_{0}\in K$ is a fixed point which is not
an endpoint. We denote the two normals to $K$ at $z_{0}$ by $n_{1}\left(z_{0}\right)$
and $n_{2}\left(z_{0}\right)$. Then there exists a sequence of polynomials
$P_{n}$ with $\deg P_{n}=n\rightarrow\infty$ such that
\begin{multline*}
\left|P_{n}'\left(z_{0}\right)\right|\ge n\left(1-o\left(1\right)\right)\left\Vert P_{n}\right\Vert _{K}\\
\cdot\max\left(\frac{\partial}{\partial n_{1}\left(z_{0}\right)}g_{\bC_{\infty}\setminus K}\left(z_{0},\infty\right),\frac{\partial}{\partial n_{2}\left(z_{0}\right)}g_{\bC_{\infty}\setminus K}\left(z_{0},\infty\right)\right).
\end{multline*}

\end{thm}

\section{A rational inequality on the unit circle}

The following theorem was proved in \cite{MR1433285} (see also \cite{MR1367960},
p. 324, Theorem 7.1.7), with  slightly different notations.

If $f$ is a rational function, then $\deg\left(f\right)$ denotes
the maximum of the degrees of the numerator and denominator of $f$
(where we assume that the numerator and the denominator have no common
factors).
\begin{thm*}
[Borwein-Erdélyi] Let $a_{1},\dots,a_{m}\in\bC\setminus\left\{ |u|=1\right\} $
and let 
\[
B_{m}^{+}\left(u\right):=\sum_{j:\;|a_{j}|>1}\frac{|a_{j}|^{2}-1}{|a_{j}-u|^{2}},\qquad B_{m}^{-}\left(u\right):=\sum_{j:\;|a_{j}|<1}\frac{1-|a_{j}|^{2}}{|a_{j}-u|^{2}},
\]
and $B_{m}\left(u\right):=\max\left(B_{m}^{+}\left(u\right),\, B_{m}^{-}\left(u\right)\right)$.
If $R$ is a polynomial with $\deg(R)\le m$ and
$f\left(u\right)=R\left(u\right)/\prod_{j=1}^{m}\left(u-a_{j}\right)$
is a rational function, then
\[
|f'\left(u\right)|\le B_{m}\left(u\right)||f||_{\bT},\qquad u\in\bT.
\]

\end{thm*}
If all the poles of $f$ are inside or outside
of $\bD$, then this result was improved in \cite{MR1332889}, Theorem
2 and Corollary 2 on page 525 using different approach.

We need to relax the condition on the degree of the numerator and
the denominator. 

If we could allow poles at infinity, then the degree of the
numerator can be larger than that of the denominator. More precisely,
we can easily obtain the following
\begin{thm}
Using the notations from Borwein-Erdélyi Theorem, 
if $R$ is a polynomial with $\deg(R)>m$ and
$f\left(u\right)=R\left(u\right)/\prod_{j=1}^{m}\left(u-a_{j}\right)$
is a rational function, then
\begin{equation}
|f'\left(u\right)|\le\max\left(B_{m}^{+}\left(u\right)+\deg\left(R\right)-m,\, B_{m}^{-}\left(u\right)\right)||f||_{\bT},\qquad w\in\bT.\label{eq:rat_ineq_suff}
\end{equation}
\end{thm}
\begin{proof}
Let $d:=\deg\left(R\right)-m>0$,
and let $f_{1}\left(\tau;\; u\right)=f_{1}\left(u\right):=\frac{f\left(u\right)}{\left(u-\tau\right)^{d}}$,
where $\tau\in\bR$, $\tau>1$. Then $\left(\tau-1\right)^{d}|f_{1}\left(u\right)|\le|f\left(u\right)|\le\left(\tau+1\right)^{d}|f_{1}\left(u\right)|$
for $|u|=1$, so 
\[
||f_{1}||_{\bT}\le\frac{1}{\left(\tau-1\right)^{d}}||f||_{\bT}.
\]
Since $f_{1}'\left(u\right)=f'\left(u\right)\frac{1}{\left(u-\tau\right)^{d}}-d\; f\left(u\right)\frac{1}{\left(u-\tau\right)^{d+1}}$,
therefore
\[
|f_{1}'\left(u\right)|\ge|f'\left(u\right)|\frac{1}{\left(\tau+1\right)^{d}}-d||f||_{\bT}\frac{1}{\left(\tau-1\right)^{d+1}}.
\]
Using Borwein-Erdélyi Theorem for $f_{1}$, $|u|=1$, 
\[
|f_{1}'\left(u\right)|\le\max\left(B_{m}^{+}\left(u\right)+d\frac{\tau^{2}-1}{|u-\tau|^{2}},\; B_{m}^{-}\left(u\right)\right)||f_{1}||_{\bT}.
\]
Letting $\tau\rightarrow\infty$ and combining the last three displayed
estimates, we obtain the Theorem.
\end{proof}
Note that if we let all the poles tend to infinity, then we get back
the original Bernstein (Riesz) inequality for polynomials on the unit
disk. Let us also remark that the original proof of Borwein and Erdélyi
also proves (\ref{eq:rat_ineq_suff}), with little modifications.

The relation with Green's functions is as follows. It is well known
(see e.g. \cite{MR1485778}, p.109) that Green's function of the unit
disk $\bD$ with pole at $a\in\bD$ is 
\[
g_{\bD}\left(u,a\right)=\log\left|\frac{1-\overline{a}u}{u-a}\right|
\]
and Green's functions of the complement of the unit disk $\bD^{*}=\left\{ |u|>1\right\} \cup\left\{ \infty\right\} $
with pole at $a\in\bC$, $|a|>1$ and with pole at infinity are 
\[
g_{\bD^{*}}\left(u,a\right)=\log\left|\frac{1-\overline{a}u}{u-a}\right|\mbox{ and }g_{\bD^{*}}\left(u,\infty\right)=\log|u|.
\]
For the normal derivatives elementary calculations give ($|u|=1$,
$n_{1}\left(u\right)=-u$ is the inner normal, $n_{2}\left(u\right)=u$
is the outer normal) 
\begin{gather}
\frac{\partial}{\partial n_{1}\left(u\right)}g_{\bD}\left(u,a\right)=\lim_{t\rightarrow0+}\frac{\log\left|\frac{1-\overline{a}\left(1-t\right)u}{\left(1-t\right)u-a}\right|}{t}=\frac{1-|a|^{2}}{|u-a|^{2}},\label{eq:ndgd}\\
\frac{\partial}{\partial n_{2}\left(u\right)}g_{\bD^{*}}\left(u,a\right)=\lim_{t\rightarrow0+}\frac{\log\left|\frac{1-\overline{a}\left(1+t\right)u}{\left(1+t\right)u-a}\right|}{t}=\frac{|a|^{2}-1}{|u-a|^{2}},\label{eq:ndgdsf}\\
\frac{\partial}{\partial n_{2}\left(u\right)}g_{\bD^{*}}\left(u,\infty\right)=\lim_{t\rightarrow0+}\frac{\log|\left(1+t\right)u|}{t}=1.\label{eq:ndgdsi}
\end{gather}
They are also mentioned in \cite{MR2380804}, p.1739. 

\medskip{}
Using this notation, we can reformulate these last two theorems
as follows.
This is actually the result of Borwein and Erdélyi with slightly different wording.

\begin{thm}
\label{thm:genBorweinErdelyi}Let $f\left(u\right)=R\left(u\right)/Q\left(u\right)$
be an arbitrary rational function with no poles on the unit circle
where $R$ and $Q$ are polynomials. Denote the poles of $f$ on $\bCi$
by $a_{1},\dots,a_{m}\in\bCi\setminus\left\{ |u|=1\right\} $ where
each pole is repeated as many times as its order. Then, for $u\in\bT$,
\begin{multline}
|f'\left(u\right)|\le||f||_{\bT}\\
\cdot\max\left(\sum_{j:|a_{j}|<1}\frac{\partial}{\partial n_{1}\left(u\right)}g_{\bD}\left(u,a_{j}\right),\ \sum_{j:|a_{j}|>1}\frac{\partial}{\partial n_{2}\left(u\right)}g_{\bD^{*}}\left(u,a_{j}\right)\right).\label{ineq:genBorweinErdelyi}
\end{multline}

\end{thm}
Note that if $\deg\left(R\right)>\deg\left(Q\right)$, then $f$ has
a pole at $\infty$, therefore it is repeated $\deg\left(R\right)-\deg\left(Q\right)$
times and this pole at $\infty$ is taken into account in the second
term of maximum. Inequality (\ref{ineq:genBorweinErdelyi}) is sharp,
the factor on the right hand side cannot be replaced for smaller constant,
see, e.g., \cite{MR1367960}, p. 324.

\section{Mapping complement of a system of arcs onto domains bounded by Jordan
curves with rational functions}

Let $K$ be a finite union of $C^{2}$ smooth, disjoint Jordan arcs
on the complex plane, that is,
\[
K=\cup_{j=1}^{k_{0}}\gamma_{j},\mbox{ where }\gamma_{j}\cap\gamma_{k}=\emptyset,\ j\ne k.
\]
Denote the endpoints of $\gamma_{j}$ by $\zeta_{2j-1},\zeta_{2j}$,
$j=1,\dots,k_{0}$. 

We need the following Proposition to transfer our setting. Although
we will use it for one analytic Jordan arc, it can be useful for further
researches. 

After we worked out the proof, we learned that Widom developed very
similar open-up Lemma in his work, see \cite{MR0239059}, p. 205-207.
The difference is that he considers $C^{k}$ smooth arcs with Hölder
continuous $k$-th derivative (see also p. 145) while we need this
open-up technique for analytic arcs. Furthermore, there is a difference
regarding the number of poles. This is discussed after the proof.
\begin{prop}
\label{prop:open-up}There exists a rational function $F$ and a domain
$G\subset\bCi$ such that $\bC\setminus G$ is a compact set with
$k_{0}$ components, $\partial\left(\bCi\setminus G\right)=\partial G$
is union of finitely many smooth Jordan curves and $F$ is a conformal
bijection from $G$ onto $\bCi\setminus K$ with $F(\infty)=\text{\ensuremath{\infty}}$.

Furthermore, if $K$ is analytic, then $\partial G$ is analytic too.\end{prop}
\begin{proof}
First, we show that there are polynomials $R$, $Q$ such that $\deg\left(R\right)=k_{0}+1$,
$\deg\left(Q\right)=k_{0}$, 
\[
F\left(u\right):=\frac{R\left(u\right)}{Q\left(u\right)}
\]
and 
\begin{equation}
F'\left(u\right)=0\Leftrightarrow F\left(u\right)\in\left\{ \zeta_{1},\dots,\zeta_{2k_{0}}\right\} .\label{eq:fminimality}
\end{equation}
Obviously, $F'\left(u\right)=\left(R'\left(u\right)Q\left(u\right)-R\left(u\right)Q'\left(u\right)\right)/Q^{2}\left(u\right)$
and the numerator is a polynomial of degree $2k_{0}$. Let $A\left(u\right):=\prod_{j=1}^{2k_{0}}\left(u-\zeta_{j}\right)$.
Taking reciprocal, $1/F'=Q^{2}/A$, that is, the location of the poles
are known. Our goal is to find $\beta_{0},\beta_{1},\beta_{2},\ldots,\beta_{2k_{0}}\in\bC$
such that 
\[
\int\frac{1}{\beta_{0}+\sum_{j=1}^{2k_{0}}\frac{\beta_{j}}{u-\zeta_{j}}}\, du\ \mbox{is a rational function.}
\]
Or equivalently, $F_{1}\left(u\right):=\frac{\prod_{k}\left(u-\zeta_{k}\right)}{\beta_{0}\prod_{k}\left(u-\zeta_{k}\right)+\sum_{j>0}\beta_{j}\prod_{k\ne j}\left(u-\zeta_{k}\right)}$
must have $0$ residue everywhere, $\res\left(F_{1},u\right)=0$ for
all $u\in\bC$. Since $\zeta_{k}$'s are pairwise different, $\prod_{k\ne j}\left(u-\zeta_{k}\right)$,
$j=1,2,\ldots,2k_{0}$ and $\prod_{k}\left(u-\zeta_{k}\right)$ are
linearly independent, so we can choose $\beta_{j}$'s so that 
\[
\beta_{0}\prod_{k}\left(u-\zeta_{k}\right)+\sum_{j>0}\beta_{j}\prod_{k\ne j}\left(u-\zeta_{k}\right)=\left(u-u^{*}\right)^{2k_{0}}
\]
where $u^{*}$ will be specified later. Write $A\left(u\right)=\prod_{k}\left(u-\zeta_{k}\right)$
in the form $A\left(u\right)=\sum_{j=0}^{2k_{0}}c_{j}\left(u-u^{*}\right)^{j}$
with suitable $c_{j}$'s. It is easy to see that $\res\left(F_{1},u\right)=0$
for all $u\ne u^{*}$, furthermore $\res\left(F_{1},u^{*}\right)=c_{2k_{0}-1}$.
Comparing the coefficients of $A\left(u\right)$, we obtain $c_{2k_{0}}=1$,
$c_{2k_{0}-1}=-\left(\sum_{j=1}^{2k_{0}}\zeta_{j}\right)+2k_{0}u^{*}$.
Rearranging the expression for $c_{2k_{0}-1}$, $u^{*}$ must satisfy
the following equation
\[
u^{*}=\frac{\sum_{j=1}^{2k_{0}}\zeta_{j}}{2k_{0}}.
\]
With this choice, there exists $F=\int F_{1}$ with the desired properties.

\smallskip{}

The domain $G$ is constructed as follows. Denote the unbounded component
of $F^{-1}\left[\bCi\setminus K\right]$ by $G$. We prove that $G$
is a domain and its boundary consists of finitely many Jordan curves
and those curves are smooth. Locally, if $z\in\gamma_{j}$ for some
$\gamma_{j}$ and $z$ is not endpoint of $\gamma_{j}$, then, by
the construction, $z$ is not a critical value. In other words, for
any $u$ such that $F\left(u\right)=z$, we know $F'\left(u\right)\ne0$
($u$ is not a critical place). If $z\in\gamma_{j}$ is an endpoint
and $u_{1}$ is any of its inverse image, then $F'\left(u_{1}\right)=0$
by (\ref{eq:fminimality}) and since the degree of $R$ and $Q$ are
minimal, $F''\left(u_{1}\right)\ne0$. Therefore $F\left(u\right)\approx c\left(u-u_{1}\right)^{2}+z$,
and the inverse image $F^{-1}\left[\gamma_{j}\right]$ of $\gamma_{j}$
near $u_{1}$ is a smooth, simple arc. So each bounded component of
$\bC\setminus G$ is such a compact set that it is a closure of a
Jordan domain. 

Using continuity and connectedness, $\bC_{\infty}\setminus F^{-1}\left[\bC_{\infty}\setminus K\right]$
has at least $k_{0}$ bounded components. If there were more than
$k_{0}$ components, then we obtain contradiction as follows. The
boundary of each component is mapped into $K$, so there should be
more than $2k_{0}$ critical points, but this contradicts the minimality
of $F$. Denote the boundary of the components by $\kappa_{j}$, $j=1,\ldots,k_{0}$.
These $\kappa_{j}$'s are smooth Jordan curves and assume $\kappa_{j}=\kappa_{j}\left(t\right)$,
$t\in\left[0,2\pi\right]$.

It is clear that each component has nonempty interior and contains
at least one pole of $F$, otherwise $F$ maps that component onto
some open, bounded, nonempty set and this set would intersect $\bC_{\infty}\setminus K$.
Therefore each component contains exactly one pole which is simple
by the minimality assumption.

Now, $F=R/Q$ is univalent on $G$ because of the followings. Take
smooth Jordan curves $\kappa_{j,\delta}\left(t\right)$, $t\in\left[0,2\pi\right]$
satisfying the next properties: $\kappa_{j,\delta}\subset G$, $\kappa_{j,\delta}\left(t\right)\rightarrow\kappa_{j}\left(t\right)$
as $\delta\rightarrow0$ and $\kappa_{j,\delta}'\left(t\right)\rightarrow\kappa_{j}'\left(t\right)$
as $\delta\rightarrow0$ and $\kappa_{0,\delta}\left(t\right):=1/\delta\exp\left(it\right)$.
Since $\deg\left(R\right)=\deg\left(Q\right)+1$, $F\left(u\right)=c_{1}u+c_{0}+o\left(1\right)$
as $u\rightarrow\infty$ therefore $F\left(\kappa_{0,\delta}\left(t\right)\right)\rightrightarrows\infty$
as $\delta\rightarrow0$ and, by continuity, $\dist\left(F\left(\kappa_{j,\delta}\right),\gamma_{j}\right)\rightarrow0$.
Since $F$ has no critical values outside $K$, the $F\left(\kappa_{j,\delta}\right)$'s
are smooth Jordan curves. Fix $b\in\bC\setminus K$, then there is
(at least one) $b'\in G$ with $F\left(b'\right)=b$, because $F\left(G\right)$
is open, $F\left(G\right)\subset\bC\setminus K$ and $F\left(\partial G\right)=F\left(\kappa_{1}\cup\ldots\cup\kappa_{k_{0}}\right)\subset K$.
If $\delta>0$ is small enough, then $b\in\inter F\left(\kappa_{0,\delta}\right)$
and $b\in\bC\setminus\inter F\left(\kappa_{j,\delta}\right)$ ($j=1,\ldots,k_{0}$),
so $\mathrm{index}\left(b,F\left(\kappa_{0,\delta}\right)\cup F\left(\kappa_{1,\delta}\right)\cup\ldots\cup F\left(\kappa_{k_{0},\delta}\right)\right)=1$.
Therefore $\mathrm{index}\left(b',\kappa_{0,\delta}\cup\kappa_{1,\delta}\cup\ldots\cup\kappa_{k_{0},\delta}\right)=1$,
so there is exactly one inverse image, this shows the univalence of
$F$.

We can give another proof for the univalence as follows. There is
a (local) branch of $F^{-1}$ such that $F^{-1}\left[z\right]=z/c_{1}+..$
as $z\rightarrow\infty$, in other words, $\infty$ is not a branch
point of $F^{-1}$. Furthermore, the function $F$ has branch points
only at $\zeta_{j}$'s, $j=1,\ldots,2k_{0}$ and it behaves as a square
root there. Therefore every analytic continuations along any curve
in $\bC\setminus K$ give the same function element. Now we use Lemma
2, p. 175 in \cite{MR1007599} with this (local) branch. Therefore
we can choose a (global) regular branch of $F^{-1}$ such that $F^{-1}\left[\infty\right]=\infty$.
Since this branch is regular and $F$ is a rational function, there
is no other inverse image of $\infty$ by $F^{-1}$ in $G$. By the
construction of $G$ and applying the maximum principle, we have $g_{\bCi\setminus K}\left(F\left(u\right),\infty\right)\equiv g_{G}\left(u,\infty\right)$,
$u\in G$. Using the majorization principle (see \cite{MR2765937},
Theorem 1 on p. 624) or Theorem 4.4.1 on p. 112 from \cite{MR1334766},
we obtain that $F$ is conformal bijection from $G$ onto $\bCi\setminus K$. 

As for the smoothness assertion ($\partial G$ analytic), this follows
from standard considerations as follows. Without loss of generality,
we may assume that $z=\kappa\left(t\right)=t+c_{1}t+c_{2}t^{2}+\ldots$,
is a convergent power series for $0\le t\le t_{0}$ and $z=F\left(u\right)$
is such that $F\left(0\right)=0$, $F'\left(0\right)=0$ and $F''\left(0\right)\ne0$.
It is known, see e.g. \cite{Stoilow}, p. 286, that the two branches
of the inverse of $F$ near $z=0$ can be written as $G_{0}\left(z\right)\pm\sqrt{z}G_{1}\left(z\right)$
where $G_{0},G_{1}$ are holomorphic functions. Denote them by $F_{1}^{-1}$
and $F_{2}^{-1}$. This way $\gamma_{1}\left(t\right):=F_{1}^{-1}\left[\kappa\left(t^{2}\right)\right]=G_{0}\left(\kappa\left(t^{2}\right)\right)+t\sqrt{1+\kappa_{1}\left(t^{2}\right)}G_{1}\left(\kappa\left(t^{2}\right)\right)$
is a convergent power series in $t\in\left[0,t_{1}\right]$ and similarly
for $\gamma_{2}\left(t\right):=F_{2}^{-1}\left[\kappa\left(t^{2}\right)\right]$
and $\gamma_{1}'\left(0\right)\ne0$. Considering $\gamma_{1}\left(-t\right)$
for $t\in\left[0,t_{1}\right]$, we see that $\gamma_{2}\left(t\right)=\gamma_{1}\left(-t\right)$,
so $\gamma_{1}$ is actually a convergent power series and it parametrizes
the two joining arc.
\end{proof}
As for the number of poles, Widom's open-up mapping is constructed
as iterating the Joukowskii mapping (composed with a suitable linear
mapping in each step) for each arc and that open-up mapping has $2^{k_{0}}$
different, simple poles and the location of poles also depends on
the order of arcs.
In contrast, our open-up rational function has $k_0$ simple poles.

\bigskip{}

With this Proposition, we switch from polynomials on Jordan arcs to
rational functions on Jordan curves as follows. We use the following
notations, assumptions.

\bigskip{}

Fix one, $C^{2}$ smooth Jordan arc $\gamma$ with endpoints $\zeta_{1}$
and $\zeta_{2}$ and let $z\in\gamma$, $z\ne\zeta_{1}$, $z\ne\zeta_{2}$.
Denote the two normal vectors of unit length at $z$ to $\gamma$
by $n_{1}\left(z\right)$, $n_{2}\left(z\right)$, where $n_{1}\left(z\right)=-n_{2}\left(z\right)$.
We may assume that $n_{1}$ and $n_{2}$ depend continuously on $z$.
We use the same letter for normals in different planes and from the
context, it is always clear that which arc we refer to. We use
the rational mapping $F$ and the domain $G_{2}:=G$ from the previous
Proposition for $\gamma$. Denote the inward normal vector to $\partial G$
at $u\in\partial G$ by $n_{2}\left(u\right)$ and the outward normal
vector to $\partial G$ at $u$ by $n_{1}\left(u\right)$, $n_{2}\left(u\right)=-n_{1}\left(u\right)$.
It is easy to see that there are two inverse images of $z$: $u_{1}=u_{1}\left(z\right),\ u_{2}=u_{2}\left(z\right)\in\partial G$
(such that $F\left(u_{1}\right)=F\left(u_{2}\right)=z$) and we can
assume that $u_{1},u_{2}$ are continuous functions of $z.$ 

By reindexing $u_{1}$ and $u_{2}$, we may assume that the normal
vector $n_{2}\left(u_{1}\right)$ is mapped by $F$ to the normal
vector $n_{2}\left(z\right)$. This immediately implies that $n_{1}\left(u_{1}\right)$,
$n_{2}\left(u_{2}\right)$, $n_{1}\left(u_{2}\right)$ are mapped
by $F$ to $n_{1}\left(z\right)$, $n_{1}\left(z\right)$, $n_{2}\left(z\right)$
respectively.

\begin{figure}
\begin{centering}
\includegraphics[width=1\textwidth]{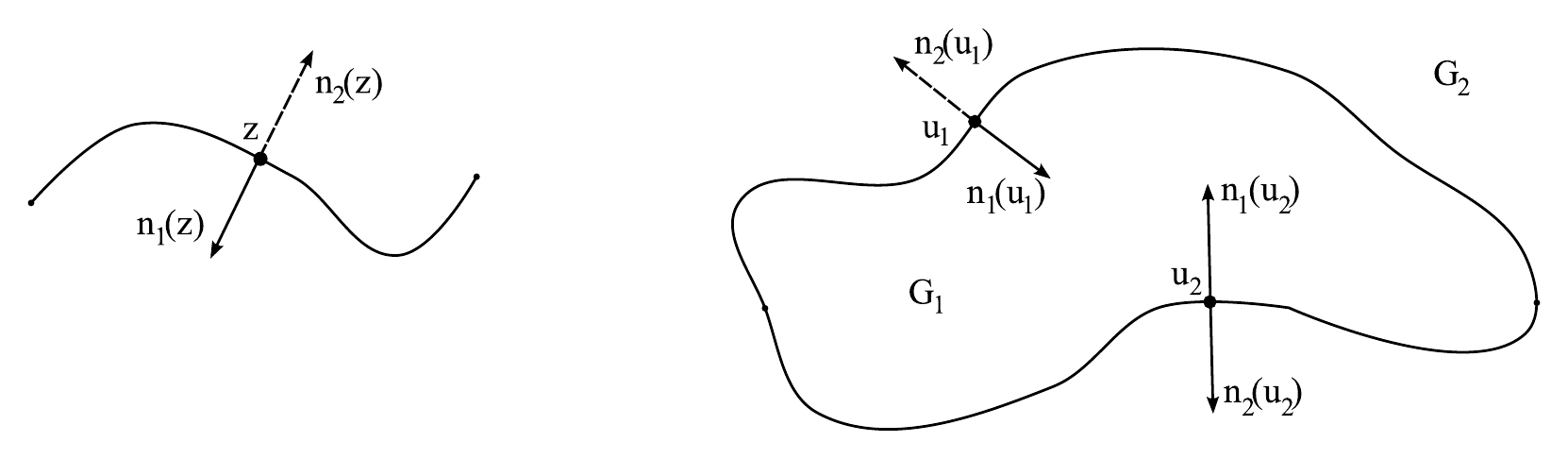}
\par\end{centering}

\protect\caption{The $\gamma$, $z$, $G_{1}$ and $G_{2}$ with the normal vectors}
\label{fig:onearconecurve}

\end{figure}

Let us denote the domain $\bC\setminus\left(G\cup\partial G\right)$
by $G_{1}$. Since $\deg F=2$ and $F$ is a conformal bijection from
$G_{2}$ onto $\bC_{\infty}\setminus\gamma$, $F$ is a conformal
bijection from $ $$G_{1}$ onto $\bC_{\infty}\setminus\gamma$. For
simplicity, let us denote the inverse of $F$ onto $G_{1}$ by $F_{1}^{-1}$
and onto $G_{2}$ by $F_{2}^{-1}$.

These geometrical objects are depicted in Figure \ref{fig:onearconecurve}
where we indicated the normal vectors $n_{2}\left(z\right)$ and $n_{2}\left(u_{1}\right)$
with dashed arrows (we fix the notations with their help) and we indicated
the other normal vectors with simple (not dashed) arrows (their indexings
are consequence of the earlier two vectors).
\begin{prop}
\label{prop:Green_trf1}Using the notations above, for the Green's
functions of $G=G_{2}$ and $G_{1}$ and for $b\in\bC_{\infty}\setminus K$
we have
\begin{multline*}
\frac{\partial}{\partial n_{1}\left(z\right)}g_{\bCi\setminus K}\left(z,b\right)=\frac{\partial}{\partial n_{1}\left(u_{1}\right)}g_{G_{1}}\left(u_{1},F_{1}^{-1}\left(b\right)\right)/\left|F'\left(u_{1}\right)\right|\\
=\frac{\partial}{\partial n_{2}\left(u_{2}\right)}g_{G_{2}}\left(u_{2},F_{2}^{-1}\left(b\right)\right)/\left|F'\left(u_{2}\right)\right|
\end{multline*}
and, similarly for the other side,
\begin{multline*}
\frac{\partial}{\partial n_{2}\left(z\right)}g_{\bCi\setminus K}\left(z,b\right)=\frac{\partial}{\partial n_{1}\left(u_{2}\right)}g_{G_{1}}\left(u_{2},F_{1}^{-1}\left(b\right)\right)/\left|F'\left(u_{2}\right)\right|\\
=\frac{\partial}{\partial n_{2}\left(u_{1}\right)}g_{G_{2}}\left(u_{1},F_{2}^{-1}\left(b\right)\right)/\left|F'\left(u_{1}\right)\right|.
\end{multline*}

For arbitrary polynomial $P$, let $f_{P}\left(u\right)=f\left(u\right):=P\left(F\left(u\right)\right)$.
Then $\left\Vert P\right\Vert _{\gamma}=\left\Vert f\right\Vert _{\partial G}$.\end{prop}
\begin{proof}
This immediately follows from the conformal invariance of Green's
functions
\[
g_{\bC_{\infty}\setminus K}\left(F\left(u\right),b\right)=g_{G_{1}}\left(u,F_{1}^{-1}\left(b\right)\right)
\]
and 
\[
g_{\bC_{\infty}\setminus K}\left(F\left(u\right),b\right)=g_{G_{2}}\left(u,F_{2}^{-1}\left(b\right)\right).
\]
See e.g. \cite{MR1334766}, p. 107, Theorem 4.4.4.
\end{proof}
This Proposition implies that it is enough to take into account the
normal derivatives at, say, $u_{1}$ only , i.e. $\frac{\partial}{\partial n_{2}\left(u_{1}\right)}g_{G_{2}}\left(u_{1},F_{2}^{-1}\left(b\right)\right)$
and $\frac{\partial}{\partial n_{1}\left(u_{1}\right)}g_{G_{1}}\left(u_{1},F_{1}^{-1}\left(b\right)\right)$
only.

\section{Conformal mappings on simply connected domains}

Here $G_{1}$ is the bounded domain from the previous section and
$G_{2}$ is the unbounded domain from the previous section. Actually,
$G_{2}=\bCi\setminus\left(G_{1}^{-}\right)$. As earlier, $\bD=\left\{ v:\ \left|v\right|<1\right\} $
and $\bD^{*}=\left\{ v:\ \left|v\right|>1\right\} \cup\left\{ \infty\right\} $.
With these notations, $\partial G_{1}=\partial G_{2}$. Using Kellogg-Warschawski
theorem (see e.g. \cite{MR1217706} p. 49, Theorem 3.6), if the boundary
is $C^{1,\alpha}$ smooth, then the Riemann mappings of $\bD,\bD^{*}$
onto $G_{1},G_{2}$ respectively and their derivatives can be extended
continuously to the boundary.

Under analyticity assumption, we can compare the Riemann mappings
as follows.
\begin{prop}
\label{prop:green_deriv_g_one_g_two}Let $u_{0}\in\partial G_{1}=\partial G_{2}$
be fixed. Then there exist two Riemann mappings $\Phi_{1}:\bD\rightarrow G_{1}$,
$\Phi_{2}:\bD^{*}\rightarrow G_{2}$ such that $\Phi_{j}\left(1\right)=u_{0}$
and $\left|\Phi_{j}'\left(1\right)\right|=1$, $j=1,2$. 

If $\partial G_{1}=\partial G_{2}$ is analytic, then there exist
$0\le r_{1}<1<r_{2}\le\infty$ such that $\Phi_{1}$ extends to $D_{1}:=\left\{ v:\ \left|v\right|<r_{2}\right\} $,
$G_{1}^{+}:=\Phi_{1}\left(D_{1}\right)$ and $\Phi_{1}:D_{1}\rightarrow G_{1}^{+}$
is a conformal bijection, and similarly, $\Phi_{2}$ extends to $D_{2}:=\left\{ v:\ \left|v\right|>r_{1}\right\} \cup\left\{ \infty\right\} $,
$G_{2}^{+}:=\Phi_{2}\left(D_{2}\right)$ and $\Phi_{2}:D_{2}\rightarrow G_{2}^{+}$
is a conformal bijection. \end{prop}
\begin{proof}
The existence of $\Phi_{1}$ follows immediately from the Riemann
mapping theorem by considering arbitrary Riemann mapping and composing
this mapping with a suitable rotation and hyperbolic translation toward
$1$ (that is, $\chi_t\left(z\right)=\left(z-t\right)/\left(1-tz\right)$
with $t\in\left(-1,1\right)$ and $t\rightarrow-1$, $\chi_t'\left(1\right)\rightarrow0$,
and $t\rightarrow1$, $\chi_t'\left(1\right)\rightarrow+\infty$).

The existence of $\Phi_{2}$ follows the same way, using the same
family of hyperbolic translations.

The extension follows from the reflection principle for analytic curves
(see e.g. \cite{MR1344449} pp. 16-21).
\end{proof}
From now on, we fix such two conformal mappings and let $a_{1}:=\Phi_{1}^{-1}\left[F_{1}^{-1}\left[\infty\right]\right]$
and $a_{2}:=\Phi_{2}^{-1}\left[\infty\right]=\Phi_{2}^{-1}\left[F_{2}^{-1}\left[\infty\right]\right]$.

The domains of these analytic extensions are depicted on Figure \ref{fig:twoRiemann}
where $D_{1}$ is the grey region on the right and is mapped onto
$G_{1}^{+}$ by $\Phi_{1}$ which is the grey region on the left.

\begin{figure}
\begin{centering}
\includegraphics[width=1\textwidth]{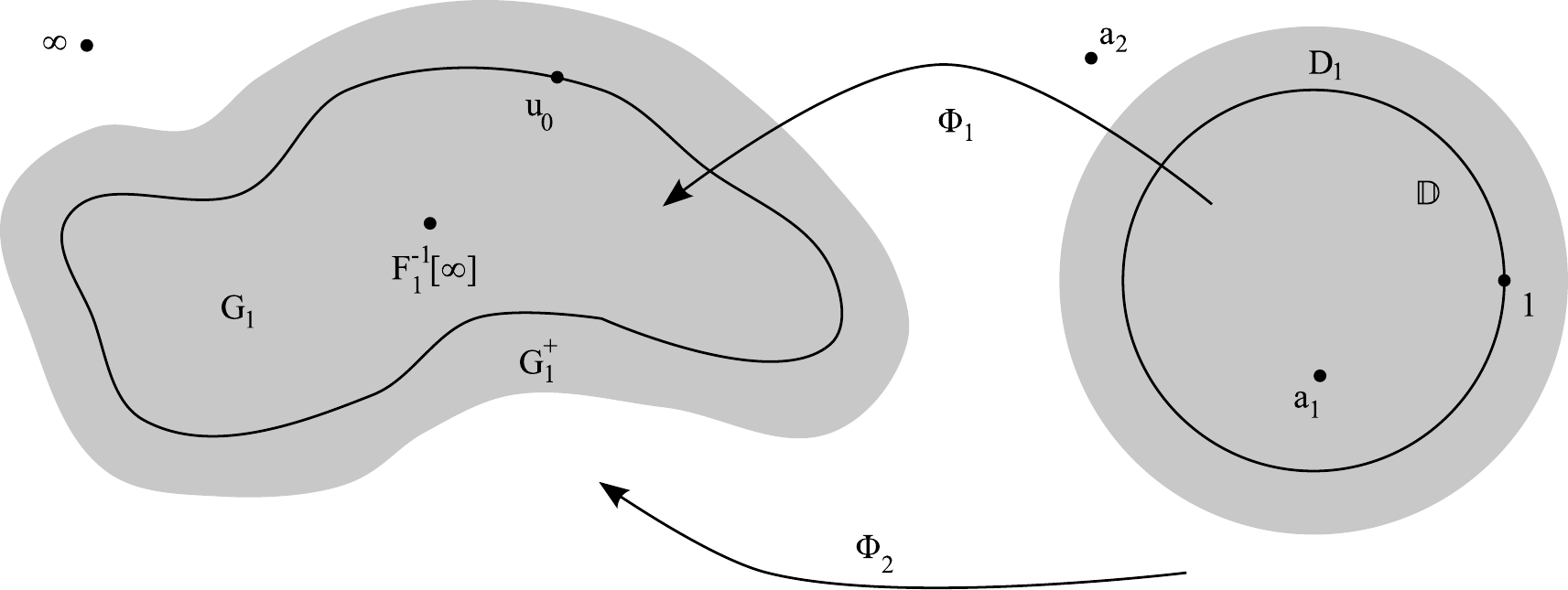}
\par\end{centering}

\protect\caption{The two Riemann mappings and the points}
\label{fig:twoRiemann}
\end{figure}

Using these mappings, we have the following relations between the
normal derivatives of Green's functions and Blaschke factors.
\begin{prop}
\label{prop:green_deriv_unitdisk}The followings hold
\begin{gather*}
\frac{\partial}{\partial n_{1}\left(u_{0}\right)}g_{G_{1}}\left(u_{0},F_{1}^{-1}\left[\infty\right]\right)=\frac{\partial}{\partial n_{1}\left(1\right)}g_{\bD}\left(1,a_{1}\right)=\frac{1-\left|a_{1}\right|^{2}}{\left|1-a_{1}\right|^{2}},\\
\frac{\partial}{\partial n_{2}\left(u_{0}\right)}g_{G_{2}}\left(u_{0},F_{2}^{-1}\left[\infty\right]\right)=\frac{\partial}{\partial n_{2}\left(1\right)}g_{\bD^{*}}\left(1,a_{2}\right)=\frac{\left|a_{2}\right|^{2}-1}{\left|1-a_{2}\right|^{2}},
\end{gather*}
and if $a_{2}=\infty$, then
\[
\frac{\partial}{\partial n_{2}\left(u_{0}\right)}g_{G_{2}}\left(u_{0},F_{2}^{-1}\left[\infty\right]\right)=\frac{\partial}{\partial n_{2}\left(1\right)}g_{\bD^{*}}\left(1,\infty\right)=1.
\]
\end{prop}
\begin{proof}
The second equalities in all three lines follow from (\ref{eq:ndgd}),
(\ref{eq:ndgdsf}) and (\ref{eq:ndgdsi}).

We know that $\Phi_{1}\left(1\right)=u_{0}$ and $\Phi_{2}\left(1\right)=u_{0}$,
moreover $\left|\Phi_{1}'\left(1\right)\right|=1$, $\left|\Phi_{2}'\left(1\right)\right|=1$
imply that $n_{j}\left(1\right)$ is mapped to $n_{j}\left(u_{0}\right)$
by $\Phi_{j}$, $j=1,2$ and the mappings $\Phi_{j}$, $j=1,2$ also
preserve the length at $1$ (there is no magnifying factor $\left|\Phi_{j}'\left(1\right)\right|^{-1}$
unlike at Proposition \ref{prop:Green_trf1}). Using the conformal
mappings $\Phi_{1}$ and $\Phi_{2}$, and the conformal invariance
of Green's functions, we obtain the first equalities in all three
lines.
\end{proof}

\section{Proof of Theorem \ref{thm:main} with rational functions}

\subsection{Auxiliary results, some notations}

Before we start the proof, let us recall three results. The first
one is Gonchar-Grigorjan estimate when we have one pole only. See
\cite{MR0417417}, Theorem 2 on p. 572 (in the english translation). 
\begin{thm*}
Let $D_{G}\subset\bC$ be a simply connected domain and its boundary
is $C^{1}$ smooth. Let $f_{G}:D_{G}\rightarrow\bC_{\infty}$ be a
meromorphic function on $D_{G}$ such that it has only one pole. Assume
that $f_{G}$ can be extended continuously to the boundary $\partial D_{G}$
of $D_{G}$. Denote $f_{G,r}$ the principal part of $f_{G}$ in $D_{G}$
(with $f_{G,r}\left(\infty\right)=0$) and let $f_{G,h}$ denote the
holomorphic part of $ $$f_{G}$ in $D_{G}$. Denote the order of
the pole of $f_{G}$ by $n_{G}$. Then $f_{G}=f_{G,r}+f_{G,h}$ and
there exists $C_{1}\left(D_{G}\right)>0$ depending on $D_{G}$ only
such that 
\begin{equation}
\left\Vert f_{G,h}\right\Vert _{\partial D_{G}}
\le 
C_{1}\left(D_{G}\right)\left(\log n_{G}+1\right)\left\Vert f_{G}\right\Vert _{\partial D_{G}}
\label{est:GoncharGrigorjan}
\end{equation}
where $\left\Vert .\right\Vert _{\partial D_{G}}$ denotes the sup
norm over the boundary of $D_{G}$. 
\end{thm*}

In the main result of this paper we are interested in asymptotics as $n\rightarrow\infty$.
In particular, if $n_G \ge 2$, then $\log n_G +1\le 3 \log\left(n_G\right)$,
so we may write $\log n_G +1 = O\left(\log n_G \right)$.

The second result is a special case of the Bernstein-Walsh estimate,
see \cite{MR1334766}, p. 156, Theorem 5.5.7 a) or \cite{MR1485778},
p. 153.
\begin{thm*}
Let $\tilde{G}\subset\bC_{\infty}$ be a domain, $\infty\in\tilde{G}$
and denote its Green's function by $g_{\tilde{G}}\left(u,\infty\right)$
with pole at infinity. Let $\tilde{f}:\tilde{G}\rightarrow\bC_{\infty}$
be a meromorphic function which has only one pole at infinity and
we denote the order of the pole by $\tilde{n}$. Assume that $\tilde{f}$
can be extended continuously to the boundary $\partial\tilde{G}$
of $\tilde{G}$. Then
\begin{equation}
\left|\tilde{f}\left(u\right)\right|\le\left\Vert \tilde{f}\right\Vert _{\partial\tilde{G}}\exp\left(\tilde{n}\, g_{\tilde{G}}\left(u,\infty\right)\right)\label{est:BernsteinWalsh}
\end{equation}
 where $\left\Vert .\right\Vert _{\partial\tilde{G}}$ denotes the
sup norm over $\partial\tilde{G}$. 
\end{thm*}
The third result is a special case of a general construction of fast
decreasing polynomials by Totik, see \cite{MR2574887}, Corollary
4.2 and Theorem 4.1 too on p. 2065.
\begin{thm*}
Let $\tilde{K}\subset\bC$ be a compact set, $\tilde{u}\in\partial\tilde{K}$
be a boundary point. Assume that $\tilde{K}$ satisfies the touching
outer-disk-condition, that is, there exists a closed disk (with positive
radius) such that its intersection with $\tilde{K}$ is $\left\{ \tilde{u}\right\} $.
Then there exist $C_{2},C_{3}>0$ such that for all $\tilde{n}$ there
exists a polynomial $\tilde{Q}$ with the following properties: $\deg\left(\tilde{Q}\right)\le\tilde{n}^{109/110}$,
$\tilde{Q}\left(\tilde{u}\right)=1$, $\left\Vert \tilde{Q}\right\Vert _{\tilde{K}}\le1$
and if $u\in\tilde{K}$, $\left|u-\tilde{u}\right|\ge\tilde{n}^{-9/10}$,
then $\left|\tilde{Q}\left(u\right)\right|\le C_{2}\exp\left(-C_{3}\tilde{n}^{1/110}\right)$.
\end{thm*}
To apply this third theorem, we introduce several notations.

We need $\psi\left(v\right):=\frac{1-\overline{a_{2}}v}{v-a_{2}}=w$
and its inverse $\psi^{-1}\left(w\right)=\frac{1+a_{2}w}{w+\overline{a_{2}}}$.
Note that $\psi\left(a_{2}\right)=\infty$, $\psi\left(1\right)=\frac{1-\overline{a_{2}}}{1-a_{2}}$
and let $b_{1}:=\frac{1-\overline{a_{2}}}{1-a_{2}}$. Obviously, $\psi\left(\partial\bD\right)=\partial\bD$.

Let $\Gamma_{1}=\left\{ w:\left|w\right|=1+\delta_{1}\right\} $ and
$\delta_{1}>0$ is chosen so that $\Gamma_{1}\subset\psi\left(D_{1}\right)$.
This $\delta_{1}$ depends on $G_{2}$ only.

Let $D_{3}:=\left\{ w:\left|w-2b_{1}\right|<1\right\} $, this disk
touches the unit disk at $b_{1}$. Fix $\delta_{2,3}^{\left(0\right)}>0$,
$ $$\delta_{2,3}^{\left(0\right)}<1$, such that $\left\{ w:\left|w\right|\le1+\delta_{2,3}^{\left(0\right)}\right\} \subset\psi\left(D_{1}\right)$.
Then for every $\delta_{2,3}\in\left(0,\delta_{2,3}^{\left(0\right)}\right]$,
$\left\{ w:\ \left|w\right|=1+\delta_{2,3}\right\} \cap\partial D_{3}$
consists of exactly two points, $w_{1}^{*}=w_{1}^{*}\left(\delta_{2,3}\right)$
and $w_{2}^{*}=w_{2}^{*}\left(\delta_{2,3}\right)$. It is easy to
see that the length of the two arcs of $\left\{ w:\left|w\right|=1+\delta_{2,3}^{\left(0\right)}\right\} $
lying in between $w_{1}^{*}$ and $w_{2}^{*}$ are different, therefore,
by reindexing them, we can assume that the shorter arc is going from
$w_{1}^{*}$ to $w_{2}^{*}$  counterclockwise. Elementary geometric
considerations show that for all $w$, $1\le\left|w\right|\le1+\delta_{2,3}$
with $\arg w\in\left\{ \arg w_{j}^{*}\left(\delta_{2,3}\right):\ j=1,2\right\} $,
we have (since $\delta_{2,3}<1$) 
\begin{equation}
\frac{1}{2}\sqrt{\delta_{2,3}}\le\left|w-b_{1}\right|\le2\sqrt{\delta_{2,3}}.\label{eq:w_star_dist}
\end{equation}

Let 
\[
K_{w}^{*}:=\left\{ w:\ \left|w\right|\le1+\delta_{2,3}^{\left(0\right)}\right\} \setminus D_{3}.
\]
Obviously, this $K_{w}^{*}$ is a compact set and satisfies the touching-outer-disk
condition at $b_{1}=\frac{1-\overline{a_{2}}}{1-a_{2}}$ of Totik's
theorem. See figure \ref{fig:gamma_two} later.

Consider 
\[
K_{u}^{*}:=\Phi_{2}\circ\psi^{-1}\left[K_{w}^{*}\cap\bD^{*}\right]\cup\Phi_{1}\circ\psi^{-1}\left[K_{w}^{*}\cap\bD^{*}\right]\cup G_{1}.
\]
This is a compact set and also satisfies the touching-outer-disk condition
at $u_{0}=\Phi_{2}\left(1\right)$ of Totik's theorem. Obviously,
$\partial G_{2}\subset K_{u}^{*}$, $G_{1}\subset K_{u}^{*}$ , $u_{0}\in K_{u}^{*}$
and if $w\in K_{w}^{*}$, then $\Phi_{1}\circ\psi^{-1}\left(w\right)\in K_{u}^{*}$
and $\Phi_{2}\circ\psi^{-1}\left(w\right)\in K_{u}^{*}$ too. Now
applying Totik's theorem, there exists a fast decreasing polynomial
for $K_{u}^{*}$ at $u_{0}$ of degree at most $n_{1}$ which we denote
by $Q=Q\left(n_{1};u\right)$. More precisely, $Q$ has the following
properties: $Q\left(u_{0}\right)=1$, $\left|Q\left(u\right)\right|\le1$
on $u\in K_{u}^{*}$, $\deg Q\le n_{1}^{109/110}\le n_{1}$ and if
$\left|u-u_{0}\right|>n_{1}^{-9/10}$, $u\in K_{u}^{*}$, then 
\begin{equation}
\left|Q\left(u\right)\right|\le C_{2}\exp\left(-C_{3}n_{1}^{1/110}\right).\label{eq:Q_fast_decreases}
\end{equation}

Let $n_{1}:=\left\lfloor \sqrt{n}\right\rfloor $, $n_{2}:=\left\lfloor n^{3/4}\right\rfloor $,
$\delta_{2,1}:=1/n$ and $\delta_{2,3}:=n^{-2/3}$.

\subsection{Proof}

In this subsection, we let $f\left(u\right):=P_{n}\left(F\left(u\right)\right)$
where $P_{n}$ is a fixed polynomial of degree $n$ %the polynomial from Theorem \ref{thm:main} 
and $F$ is the open-up rational function (see Proposition \ref{prop:open-up})
for $K$ (from Theorem \ref{thm:main}).

Actually, we use only the following facts. $f$ is a rational function
such that it has one pole in $G_{1}$ and one in $G_{2}$. We know
that the poles of $f$ are $\infty=F_{2}^{-1}\left[\infty\right]$
and $F_{1}^{-1}\left[\infty\right]$, and the order of the pole in
$G_{1}$ is $n$. 

It is easy to decompose $f$ into sum of rational functions, that
is,
\[
f=f_{1}+f_{2}
\]
where $f_{1}$ is a rational function with pole in $G_{1}$, $f_{1}\left(\infty\right)=0$
and $f_{2}$ is a polynomial (rational function with pole at $\infty$).
This decomposition is unique. We use the Gonchar-Grigorjan estimate
(\ref{est:GoncharGrigorjan}) for $f_{2}$ on $G_{1}^{+}$, so we
have
\begin{equation}
\left\Vert f_{2}\right\Vert _{\partial G_{2}}
\le 
C_{1}\left(G_{1}^{+}\right)\left(\log n + 1 \right)\left\Vert f\right\Vert _{\partial G_{2}}.\label{eq:f_two_norm_est}
\end{equation}
Obviously, we have 
\begin{equation}
\left\Vert f_{1}\right\Vert _{\partial G_{2}}
\le
\left(1+C_{1}\left(G_{1}^{+}\right)\left(\log n + 1 \right)\right)\left\Vert f\right\Vert _{\partial G_{2}}.
\label{eq:f_one_norm_est}
\end{equation}

Consider 
\[
\varphi_{1}\left(v\right):=f_{1}\left(\Phi_{1}\left(v\right)\right).
\]
This is a meromorphic function in $D_{1}$. We may assume that $\varphi_{1}$
has only one pole in $D_{1}$ otherwise we can decrease $r_{2}>1$ so that
the pole in $G_{2}$ is not in $\Phi_{1}\left(D_{1}\right)=G_{1}^{+}$. 
We know that 
\begin{equation}
\left\Vert \varphi_{1}\right\Vert _{\partial\bD}=\left\Vert f_{1}\right\Vert _{\partial G_{2}}\label{eq:phi_one_norm}
\end{equation}
 and $\left|\varphi_{1}'\left(1\right)\right|=\left|f_{1}'\left(u_{0}\right)\right|$. 

We decompose ``the essential part of'' $\varphi_{1}$ as follows
\begin{equation}
Q\circ\Phi_{1}\cdot\varphi_{1}=\varphi_{1r}+\varphi_{1e}\label{eq:phi_decomp}
\end{equation}
where $\varphi_{1r}$ is a rational function, $\varphi_{1r}\left(\infty\right)=0$
and $\varphi_{1e}$ is holomorphic in $\bD$. 
We use the Gonchar-Grigorjan
estimate (\ref{est:GoncharGrigorjan}) again for $\varphi_{1}$ on
$\bD$, this way the following sup norm estimate holds
\begin{equation}
\left\Vert \varphi_{1e}\right\Vert _{\partial\bD}
\le
 C_{1}\left(\bD\right)\left(\log n + 1 \right)\left\Vert Q\circ\Phi_{1}\cdot\varphi_{1}\right\Vert _{\partial\bD}
\le
 C_{1}\left(\bD\right)\left(\log n + 1 \right)\left\Vert \varphi_{1}\right\Vert _{\partial\bD}
\label{eq:phi_one_e_norm_est}
\end{equation}
where $C_{1}\left(\bD\right)$ is a constant independent of $\varphi_{1}$.

As a remark, let us note that we may write $\log n+1 \le O\left(\log n\right)$ for simplicity since 
we are interested in asymptotics as $n\rightarrow \infty$ in the main theorem.
Otherwise, if $n=0$ or $n=1$, then $P_n$ is a constant or linear polynomial and 
the error term $o(1)$ in the main theorem (Theorem \ref{thm:main}) can be
sufficiently large (depending on $K$ and $z_0$) for these two particular values of $n$.
In this manner, we write $(\log n+1)$ in general, but we simplify it to $O(\log n)$ frequently.

Furthermore, we can estimate $\varphi_{1e}\left(v\right)$ on $v\in D_{1}\setminus\bD$
as follows
\begin{equation}
\left|\varphi_{1e}\left(v\right)\right|=\left|\left(Q\cdot f_{1}\right)\circ\Phi_{1}\left(v\right)-\varphi_{1r}\left(v\right)\right|\le\left|\left(Q\cdot f_{1}\right)\circ\Phi_{1}\left(v\right)\right|+\left|\varphi_{1r}\left(v\right)\right|.\label{est:phi_one_e_v}
\end{equation}

We also need to estimate $Q$ outside $\bD$ (and $K_{w}^{*}$) as
follows. Using $\deg Q\le n_{1}^{109/110}\le n_{1}$ and Bernstein-Walsh
estimate (\ref{est:BernsteinWalsh}), we can write for $v\in D_{1}\setminus\bD$
\[
\left|Q\left(\Phi_{1}\left(v\right)\right)\right|\le1\cdot\exp\left(n_{1}g_{G_{2}}\left(\Phi_{1}\left(v\right),\infty\right)\right).
\]
Since the set $\Phi_{1}\left(D_{1}\setminus\bD\right)$ is bounded,
\[
C_{6}:=\sup\left\{ g_{G_{2}}\left(\Phi_{1}\left(v\right),\infty\right):\ v\in D_{1}\setminus\bD\right\} <\infty.
\]
Therefore, for all $v\in D_{1}\setminus\bD$, 
\[
\left|\left(Q\cdot f_{1}\right)\circ\Phi_{1}\left(v\right)\right|\le e^{C_{6}n_{1}}\left\Vert f_{1}\right\Vert _{\partial G_{2}}.
\]

This way we can continue (\ref{est:phi_one_e_v}) and we use $u=\Phi_{1}\left(v\right)$
here and that $\varphi_{1r}$ is a rational function with no poles
outside $\bD$ and the maximum principle for $\varphi_{1r}$ 
\[
\le e^{C_{6}n_{1}}\left|f_{1}\left(u\right)\right|+\left\Vert \varphi_{1r}\right\Vert _{\partial\bD}
\le 
e^{C_{6}n_{1}}\left\Vert f_{1}\right\Vert _{\partial G_{2}}+\left\Vert \varphi_{1}\right\Vert _{\partial\bD}+\left\Vert \varphi_{1e}\right\Vert _{\partial\bD}
\]
and here we used that $f_{1}$ has no pole in $G_{2}$ and the maximum
principle. We can estimate these three sup norms with the help of
(\ref{eq:f_one_norm_est}) and (\ref{eq:phi_one_norm}), (\ref{eq:f_one_norm_est})
and (\ref{eq:phi_one_e_norm_est}), (\ref{eq:phi_one_norm}), (\ref{eq:f_one_norm_est}).
Hence we have for $v\in D_{1}\setminus\bD$ 
\begin{multline}
\left|\varphi_{1e}\left(v\right)\right|
\le
\left(e^{C_{6}n_{1}}+1+C_{1}\left(\bD\right)\left(\log n + 1 \right)\right)
\left(1+C_{1}\left(G_{1}^{+}\right)\left(\log n + 1 \right)\right)\left\Vert f\right\Vert _{\partial G_{2}}
\\
=O\left(\log\left(n\right)e^{C_{6}n_{1}}\right)\left\Vert f\right\Vert _{\partial G_{2}}.
\label{eq:phi_one_e_v_est}
\end{multline}

Approximate and interpolate $\varphi_{1e}$ as follows with rational
function which has only one pole, namely at $a_{2}=\Phi_{2}^{-1}\left[\infty\right]$.
Consider $\varphi_{1e}\circ\psi^{-1}\left(w\right)$ on $\psi\left(D_{1}\right)$.
Using the properties of $\psi$, we have 
\[
\left\Vert \varphi_{1e}\right\Vert _{\partial\bD}=\left\Vert \varphi_{1e}\circ\psi^{-1}\right\Vert _{\partial\bD}
\]
and $\varphi_{1e}\circ\psi^{-1}$ is a holomorphic function in $\psi\left(D_{1}\right)$.
We interpolate and use integral estimates for the error, see e.g.
\cite{MR1334766}, p. 170, proof of Theorem 6.3.1 or \cite{MR0229803},
p. 11. Therefore, let 
\[
q_{N}\left(w\right):=w^{N}\left(w-b_{1}\right)^{2}
\]
where $N=n+\left\lfloor \sqrt{n}\right\rfloor +\left\lfloor n^{3/4}\right\rfloor =n\left(1+o\left(1\right)\right)$.
We define the approximating polynomial 
\[
p_{1,N}\left(w\right):=\frac{1}{2\pi i}\int_{\Gamma_{1}}\frac{\varphi_{1e}\circ\psi^{-1}\left(\omega\right)}{q_{N}\left(\omega\right)}\frac{q_{N}\left(w\right)-q_{N}\left(\omega\right)}{w-\omega}d\omega.
\]
It is well known that $p_{1,N}$ does not depend on $\Gamma_{1}$.
Since $b_{1}$ is a double pole of $q_{N}$, therefore $p_{1,N}$
and $p_{1,N}'$ coincide there with $\varphi_{1e}\circ\psi^{-1}$
and $\left(\varphi_{1e}\circ\psi^{-1}\right)'$ respectively.

The error of the approximating polynomial $p_{1,N}$ to $\varphi_{1e}\circ\psi^{-1}$
is 
\begin{multline}
\varphi_{1e}\circ\psi^{-1}\left(w\right)-p_{1,N}\left(w\right)=\frac{1}{2\pi i}\int_{\Gamma_{1}}\frac{\varphi_{1e}\circ\psi^{-1}\left(\omega\right)}{\omega-w}\frac{q_{N}\left(w\right)}{q_{N}\left(\omega\right)}d\omega\\
=\frac{1}{2\pi i}\int_{\Gamma_{1}}\frac{1}{\omega-w}\ q_{N}\left(w\right)\ \frac{\varphi_{1e}\circ\psi^{-1}\left(\omega\right)}{q_{N}\left(\omega\right)}\; d\omega,\label{eq:p_N_one_error_est}
\end{multline}
here $w\in\bD$ can be arbitrary. It is easy to see that for $w\in\bD$,
$\left|q_{N}\left(w\right)\right|\le4$ and 
\[
\frac{1}{2\pi}\int_{\Gamma_{1}}\left|\frac{1}{\omega-w}\right|\left|d\omega\right|\le\frac{1+\delta_{1}}{\delta_{1}}.
\]
Therefore, using (\ref{eq:phi_one_e_v_est}), we can estimate the
error (of approximation of $p_{1,N}$ to $\varphi_{1e}\circ\psi^{-1}$
) as follows
\begin{multline*}
\left|\varphi_{1e}\circ\psi^{-1}\left(w\right)-p_{1,N}\left(w\right)\right|\le\frac{4\left(1+\delta_{1}\right)}{\delta_{1}}O\left(\log\left(n\right)e^{C_{6}n_{1}}\right)\left\Vert f\right\Vert _{\partial G_{2}}\frac{1}{\delta_{1}^{2}\left(1+\delta_{1}\right)^{N}}\\
=\frac{4\left(1+\delta_{1}\right)}{\delta_{1}^{3}}\frac{O\left(\log\left(n\right)e^{C_{6}n_{1}}\right)}{\left(1+\delta_{1}\right)^{N}}\left\Vert f\right\Vert _{\partial G_{2}}
\end{multline*}
which tends to $0$ as $n\rightarrow\infty$, because $n_{1}=\left\lfloor \sqrt{n}\right\rfloor $
and 
\[
\frac{e^{C_{6}n_{1}}}{\left(1+\delta_{1}\right)^{N}}=\exp\left(C_{6}\sqrt{n}-\log\left(1+\delta_{1}\right)n\left(1+o\left(1\right)\right)\right)\rightarrow0.
\]
 Considering $p_{1,N}\circ\psi$, it is a rational function with pole
at $a_{2}$ only, the order of its pole at $a_{2}$ is at most $N$
and we know that 
\begin{equation}
\left\Vert \varphi_{1e}-p_{1,N}\circ\psi\right\Vert _{\partial\bD}=o\left(1\right)\left\Vert f\right\Vert _{\partial G_{2}}\label{eq:phi_one_e_sup_err_est}
\end{equation}
where $o\left(1\right)$ is independent of $P_{n}$ and $f$ and depends
only on $G_{2}$ and tends to $0$ as $n\rightarrow\infty$, furthermore
\begin{equation}
\varphi_{1e}'\left(1\right)=\left(p_{1,N}\circ\psi\right)'\left(1\right).\label{eq:phi_one_e_deriv_at_one}
\end{equation}

Now we interpolate and approximate $f_{2}\circ\Phi_{1}$. As earlier,
we do not need the full information of this function, it is enough
to deal with $f_{2}\circ\Phi_{1}$ locally around $1$ and preserve
the sup norm. Therefore we ``chop off'' ``the unnecessary parts
of $f_{2}\circ\Phi_{1}$'' with the fast decreasing polynomial $Q$.

We have the following description about the growth of Green's function. 
\begin{lem}
\label{lem:green_est}There exists $C_{4}>0$ depending on $\delta_{2,3}^{\left(0\right)}$,
that is, depending on $G_{2}$ only and is independent of $P_{n},n$
and $f$ such that for all $1\le\left|w\right|\le1+\delta_{2,3}^{\left(0\right)}$
we have
\[
\left|\frac{\left(\psi\circ\Phi_{2}^{-1}\circ\Phi_{1}\circ\psi^{-1}\right)'\left(w\right)}{\psi\circ\Phi_{2}^{-1}\circ\Phi_{1}\circ\psi^{-1}\left(w\right)}\right|\le C_{4}.
\]
 and 
\begin{equation}
g_{G_{2}}\left(\Phi_{1}\circ\psi^{-1}\left(w\right),\infty\right)\le C_{4}\left(\left|w\right|-1\right).\label{eq:green_est_away}
\end{equation}
 Furthermore, there exists $C_{5}>0$ which depends on $G_{2}$ and
independent of $P_{n},n$ and $f$ such that for all $1\le\left|\zeta\right|\le1+\delta_{2,3}^{\left(0\right)}$
we have 
\[
\left|\frac{\left(\psi\circ\Phi_{2}^{-1}\circ\Phi_{1}\circ\psi^{-1}\right)'\left(\zeta\right)}{\psi\circ\Phi_{2}^{-1}\circ\Phi_{1}\circ\psi^{-1}\left(\zeta\right)}\right|\le1+C_{5}\left|\zeta-b_{1}\right|
\]
and 
\begin{equation}
g_{G_{2}}\left(\Phi_{1}\circ\psi^{-1}\left(\zeta\right),\infty\right)\le\left(\left|\zeta\right|-1\right)\left(1+C_{5}\left|\zeta-b_{1}\right|\right).\label{eq:green_est_close}
\end{equation}
\end{lem}
\begin{proof}
For simplicity, let $\zeta^{*}:=\arg\zeta$ where $\arg\zeta=\zeta/\left|\zeta\right|$,
if $\zeta\ne0$ and $\arg0=0$. 

We can express Green's function in the following ways for $u\in G_{2}$,
\[
g_{G_{2}}\left(u,\infty\right)=\log\left|\psi\circ\Phi_{2}^{-1}\left(u\right)\right|
\]
and for $w\in\bD^{*}$
\[
g_{G_{2}}\left(\Phi_{1}\circ\psi^{-1}\left(w\right),\infty\right)=\log\left|\psi\circ\Phi_{2}^{-1}\circ\Phi_{1}\circ\psi^{-1}\left(w\right)\right|.
\]

The first displayed inequality in the Lemma comes from continuity
considerations and the conformal bijection properties. Integrating
this inequality along radial rays, we obtain (\ref{eq:green_est_away}).
If we are close to $1$, then more is true: 
\[
\left|\left(\psi\circ\Phi_{2}^{-1}\circ\Phi_{1}\circ\psi^{-1}\right)'\left(b_{1}\right)\right|=1.
\]

Using continuity, we see that there exists $C_{5}>0$ such that for
all $\zeta$, $1\le\left|\zeta\right|\le1+\delta_{2,3}^{\left(0\right)}$,
we have 
\[
\left|\frac{\left(\psi\circ\Phi_{2}^{-1}\circ\Phi_{1}\circ\psi^{-1}\right)'\left(\zeta\right)}{\psi\circ\Phi_{2}^{-1}\circ\Phi_{1}\circ\psi^{-1}\left(\zeta\right)}\right|\le1+C_{5}\left|\zeta-b_{1}\right|.
\]
In particular, for all $\eta$ from the segment $\left[\zeta^{*},\zeta\right]$,
$\eta\in\left[\zeta^{*},\zeta\right]$,
\[
\left|\frac{\left(\psi\circ\Phi_{2}^{-1}\circ\Phi_{1}\circ\psi^{-1}\right)'\left(\eta\right)}{\psi\circ\Phi_{2}^{-1}\circ\Phi_{1}\circ\psi^{-1}\left(\eta\right)}\right|\le1+C_{5}\left|\eta-b_{1}\right|
\]
and $\left|\eta-b_{1}\right|\le\left|\zeta-b_{1}\right|$. Therefore,
integrating with respect to $\eta$ along $\left[\zeta^{*},\zeta\right]$,
we obtain
\begin{multline*}
g_{G_{2}}\left(\Phi_{1}\circ\psi^{-1}\left(\zeta\right),\infty\right)=\Re\int_{\zeta^{*}}^{\zeta}\frac{\left(\psi\circ\Phi_{2}^{-1}\circ\Phi_{1}\circ\psi^{-1}\right)'\left(\eta\right)}{\psi\circ\Phi_{2}^{-1}\circ\Phi_{1}\circ\psi^{-1}\left(\eta\right)}d\eta\\
\le\int_{\zeta^{*}}^{\zeta}\left|\frac{\left(\psi\circ\Phi_{2}^{-1}\circ\Phi_{1}\circ\psi^{-1}\right)'\left(\eta\right)}{\psi\circ\Phi_{2}^{-1}\circ\Phi_{1}\circ\psi^{-1}\left(\eta\right)}\right|\left|d\eta\right|\le\int_{\zeta^{*}}^{\zeta}1+C_{5}\left|\zeta-b_{1}\right|\left|d\eta\right|\\
=\left(\left|\zeta\right|-1\right)\left(1+C_{5}\left|\zeta-b_{1}\right|\right).
\end{multline*}

\end{proof}
Now we give the approximating polynomial as follows
\[
p_{2,N}\left(w\right):=\frac{1}{2\pi i}\int_{\Gamma}\frac{\left(Q\cdot f_{2}\right)\circ\Phi_{1}\circ\psi^{-1}\left(\omega\right)}{q_{N}\left(\omega\right)}\frac{q_{N}\left(w\right)-q_{N}\left(\omega\right)}{w-\omega}d\omega
\]
where $\Gamma$ can be arbitrary with $\bD\subset\inter\Gamma$ and
$\Gamma\subset\psi\left(D_{1}\right)$. We remark that we use the
same interpolating points, but we need a different $\Gamma$ for the
error estimate.

Now we construct $\Gamma=\Gamma_{2}$ for the estimate and investigate
the error. We use $\delta_{2,1}=1/n$, $\delta_{2,3}=n^{-2/3}$ and
$n_{2}=\left\lfloor n^{3/4}\right\rfloor $. We give four Jordan arcs
that will make up $\Gamma_{2}$. Let $\Gamma_{2,3}$ be the (shorter,
circular) arc between $w_{1}^{*}\left(\delta_{2,3}\right)$ and $w_{2}^{*}\left(\delta_{2,3}\right)$,
$\Gamma_{2,1}$ be the longer circular arc between $ $$w_{1}^{*}\left(\delta_{2,3}\right)\frac{1+\delta_{2,1}}{1+\delta_{2,3}}$
and $w_{2}^{*}\left(\delta_{2,3}\right)\frac{1+\delta_{2,1}}{1+\delta_{2,3}}$,
$\Gamma_{2,2}:=\left\{ w:\ 1+\delta_{2,1}\le\left|w\right|\le1+\delta_{2,3},\ \arg w=\arg\left(w_{1}^{*}\left(\delta_{2,3}\right)\right)\right\} $
and similarly $\Gamma_{2,4}:=\left\{ w:\ 1+\delta_{2,1}\le\left|w\right|\le1+\delta_{2,3},\ \arg w=\arg\left(w_{2}^{*}\left(\delta_{2,3}\right)\right)\right\} $
be the two segments connecting $\Gamma_{2,1}$ and $\Gamma_{2,3}$.
Finally let $\Gamma_{2}$ be the union of $\Gamma_{2,1}$, $\Gamma_{2,2}$,
$\Gamma_{2,3}$ and $\Gamma_{2,4}$. 
Figure \ref{fig:gamma_two}
depicts these arcs and $K_{w}^{*}$ defined above. 

\begin{figure}
\begin{centering}
\includegraphics{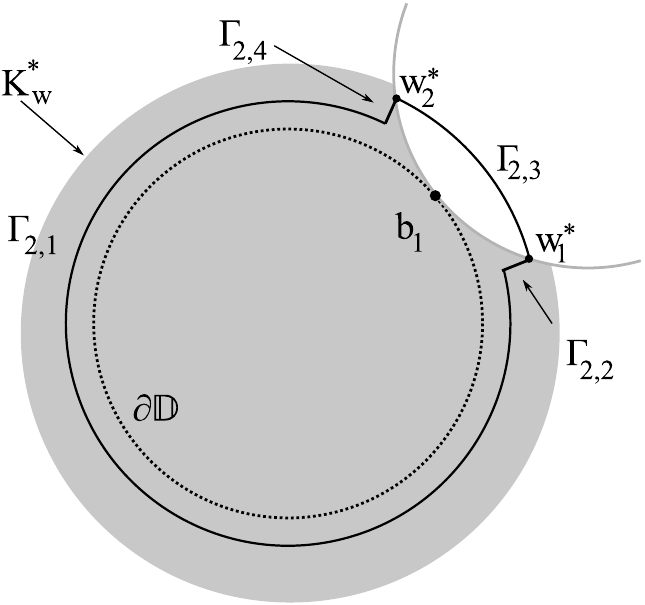}
\par\end{centering}

\protect\caption{$K_{w}^{*}$ and the arcs that make up $\Gamma_{2}$ }
\label{fig:gamma_two}

\end{figure}

We estimate the error of $p_{2,N}$ to $\left(Q\cdot f_{2}\right)\circ\Phi_{1}\circ\psi^{-1}$
on each integral separately:
\begin{multline*}
\left(Q\cdot f_{2}\right)\circ\Phi_{1}\circ\psi^{-1}\left(w\right)-p_{2,N}\left(w\right)=\frac{1}{2\pi i}\int_{\Gamma_{2}}\frac{\left(Q\cdot f_{2}\right)\circ\Phi_{1}\circ\psi^{-1}\left(\omega\right)}{\omega-w}\frac{q_{N}\left(w\right)}{q_{N}\left(\omega\right)}d\omega\\
=\frac{1}{2\pi i}\left(\int_{\Gamma_{2,1}}+\int_{\Gamma_{2,2}}+\int_{\Gamma_{2,3}}+\int_{\Gamma_{2,4}}\right).
\end{multline*}
For the first term, we use  Bernstein-Walsh estimate (\ref{est:BernsteinWalsh})
for the polynomial $f_{2}$ on $G_{2}$ and the fast decreasing polynomial
$Q$ as follows. If $w\in\Gamma_{2,1}$, then with (\ref{eq:green_est_away}),
$g_{G_{2}}\left(\Phi_{1}\circ\psi^{-1}\left(w\right),\infty\right)\le C_{4}\delta_{2,1}=C_{4}/n$,
therefore
\begin{multline*}
\left|f_{2}\left(\Phi_{1}\circ\psi^{-1}\left(w\right)\right)\right|
\le
\left\Vert f_{2}\right\Vert _{\partial G_{2}}\exp\left(n\frac{C_{4}}{n}\right)
\le
\left\Vert f\right\Vert _{\partial G_{2}}C_{1}\left(G_{1}^{+}\right)\left(\log n+1\right)e^{C_{4}}\\
=
O\left(\log\left(n\right)\right)\left\Vert f\right\Vert _{\partial G_{2}}
\end{multline*}
where we used (\ref{eq:f_two_norm_est}). Now we use the fast decreasing
property of $Q$ as follows. We know that $\Gamma_{2,1}\subset K_{w}^{*}$
(if $n\ge1/\delta_{2,3}^{\left(0\right)}$ ) and with the elementary
geometric considerations (\ref{eq:w_star_dist}) we have $\sqrt{\delta_{2,3}}/2\ge n_{1}^{-9/10}$
which is equivalent to $n^{-1/3}/2\ge n^{-9/20}$ (this is true if
$n$ is large). It is also important that $\sup\left\{ \left|\left(\Phi_{1}\circ\psi^{-1}\right)'\left(w\right)\right|:\ w\in\psi\left(D_{1}\right)\right\} <\infty$
and $K_{w}^{*}\subset\psi\left(D_{1}\right)$ therefore the growth
order of the distances is preserved by $\Phi_{1}\circ\psi^{-1}$.
Hence the fast decreasing polynomial $Q$ is small, see (\ref{eq:Q_fast_decreases}),
and we can write 
\[
\left|\left(Q\cdot f_{2}\right)\left(\Phi_{1}\circ\psi^{-1}\left(w\right)\right)\right|\le O\left(\frac{\log\left(n\right)}{\exp\left(C_{3}n^{1/220}\right)}\right)\left\Vert f\right\Vert _{\partial G_{2}}
\]
and integrating along $\Gamma_{2,1}$, we can write for $w\in\bD$
\begin{multline*}
\left|\frac{1}{2\pi i}\int_{\Gamma_{2,1}}\frac{\left(Q\cdot f_{2}\right)\circ\Phi_{1}\circ\psi^{-1}\left(\omega\right)}{\omega-w}\frac{q_{N}\left(w\right)}{q_{N}\left(\omega\right)}d\omega\right|\\
\le\frac{1}{2\pi}\int_{\Gamma_{2,1}}\frac{1}{\left|\omega-w\right|}O\left(\frac{\log\left(n\right)}{\exp\left(C_{3}n^{1/220}\right)}\right)\left\Vert f\right\Vert _{\partial G_{2}}4\frac{1}{\left(1+\delta_{2,1}\right)^{N}\delta_{2,1}^{2}}\left|d\omega\right|\\
\le\frac{2}{\pi}\frac{2\pi\left(1+\delta_{2,1}\right)}{\left(1+\delta_{2,1}\right)^{N}\delta_{2,1}^{3}}O\left(\frac{\log\left(n\right)}{\exp\left(C_{3}n^{1/220}\right)}\right)\left\Vert f\right\Vert _{\partial G_{2}}=O\left(\frac{n^{3}\log\left(n\right)}{\exp\left(C_{3}n^{1/220}\right)}\right)\left\Vert f\right\Vert _{\partial G_{2}}
\end{multline*}
here we used $\delta_{2,1}=1/n$.\medskip{}

We estimate the third term, the integral on $\Gamma_{2,3}$, as follows
for $w\in\bD$
\begin{multline}
\left|\frac{1}{2\pi i}\int_{\Gamma_{2,3}}\frac{\left(Q\cdot f_{2}\right)\circ\Phi_{1}\circ\psi^{-1}\left(\omega\right)}{\omega-w}\frac{q_{N}\left(w\right)}{q_{N}\left(\omega\right)}d\omega\right|\\
\le\frac{1}{2\pi}\int_{\Gamma_{2,3}}4\frac{1}{\left|\omega-w\right|}\left|\left(Q\cdot f_{2}\right)\left(\Phi_{1}\circ\psi^{-1}\left(\omega\right)\right)\right|\frac{1}{\left|q_{N}\left(\omega\right)\right|}\left|d\omega\right|.\label{eq:third_term_main_est}
\end{multline}
Here, $\left|\omega\right|=1+\delta_{2,3}$, $\left|w-\omega\right|\ge\delta_{2,3}$,
$\left|q_{N}\left(\omega\right)\right|\ge\delta_{2,3}^{2}\left(1+\delta_{2,3}\right)^{N}$.
Roughly speaking, $f_{2}$ grows and this time $Q$ grows too (the
bad guys) and only $\left|q_{N}\left(\omega\right)\right|^{-1}$ decreases
(the good guy). We estimate their growth using  Bernstein-Walsh
estimate (\ref{est:BernsteinWalsh}) for $f_{2}$ on $G_{2}$ and
Lemma \ref{lem:green_est} (and  estimate (\ref{eq:f_two_norm_est})
as well) in the following way. 
Here, as earlier, $\omega\in\Gamma_{2,3}$
\begin{multline*}
\left|f_{2}\left(\Phi_{1}\circ\psi^{-1}\left(\omega\right)\right)\right|
\le
\left\Vert f_{2}\right\Vert _{\partial G_{2}}\exp\left(ng_{G_{2}}\left(\Phi_{1}\circ\psi^{-1}\left(\omega\right),\infty\right)\right)
\\
\le 
C_{1}\left(G_{1}^{+}\right)\left(\log n + 1 \right)\left\Vert f\right\Vert _{\partial G_{2}}\exp\left(n\left(\left|\omega\right|-1\right)\left(1+C_{5}\left|\omega-b_{1}\right|\right)\right)
\\
\le 
C_{1}\left(G_{1}^{+}\right)\left(\log n + 1 \right)\left\Vert f\right\Vert _{\partial G_{2}}\exp\left(n\delta_{2,3}+C_{5}n\delta_{2,3}2\sqrt{\delta_{2,3}}\right)
\\
=
C_{1}\left(G_{1}^{+}\right)\left(\log n + 1 \right)\left\Vert f\right\Vert _{\partial G_{2}}\exp\left(n\delta_{2,3}\right)e^{2C_{5}}
\end{multline*}
where in the last two steps we used $\left|\omega-b_{1}\right|\le2\sqrt{\delta_{2,3}}$
from (\ref{eq:w_star_dist}) and $\delta_{2,3}=n^{-2/3}$.

As for $q_{N}$,
\begin{multline*}
\frac{1}{\left|q_{N}\left(\omega\right)\right|}\le\frac{1}{\delta_{2,3}^{2}}\frac{1}{\left(1+\delta_{2,3}\right)^{N}}=\frac{1}{\delta_{2,3}^{2}}\exp\left(-\left(n+n_{1}+n_{2}\right)\log\left(1+\delta_{2,3}\right)\right)\\
\le\frac{1}{\delta_{2,3}^{2}}\exp\left(-n\delta_{2,3}-n_{1}\delta_{2,3}-n_{2}\delta_{2,3}+\left(n+n_{1}+n_{2}\right)\frac{\delta_{2,3}^{2}}{2}\right)\\
\le\frac{1}{\delta_{2,3}^{2}}\exp\left(-n\delta_{2,3}-n_{1}\delta_{2,3}-n_{2}\delta_{2,3}\right)\exp\left(3n\ n^{-4/3}\right)\\
\le\frac{\exp\left(-n\delta_{2,3}-n_{1}\delta_{2,3}-n_{2}\delta_{2,3}\right)}{\delta_{2,3}^{2}}e^{3}
\end{multline*}
where we used $n_{1}=\left\lfloor n^{1/2}\right\rfloor $, $n_{2}=\left\lfloor n^{3/4}\right\rfloor $
and $\delta_{2,3}=n^{-2/3}$.

As for $Q$ (this time it is a bad guy), we use  Bernstein-Walsh
estimate (\ref{est:BernsteinWalsh}) for $Q$ on $G_{1}\cup\partial G_{1}$
and that $G_{1}\cup\partial G_{1}\subset K_{u}^{*}$. Therefore, $\left\Vert Q\right\Vert _{\partial G_{2}}=1$
and we know that $\deg Q\le n_{1}^{109/110}\le n^{109/220}$, hence
\begin{multline*}
\left|Q\left(\Phi_{1}\circ\psi^{-1}\left(\omega\right)\right)\right|\le\left\Vert Q\right\Vert _{\partial G_{2}}\exp\left(n_{1}g_{G_{2}}\left(\Phi_{1}\circ\psi^{-1}\left(\omega\right),\infty\right)\right)\\
\le\exp\left(n_{1}\left(\left|\omega\right|-1\right)\left(1+C_{5}\left|\omega-b_{1}\right|\right)\right)\le\exp\left(n^{109/220}\delta_{2,3}\left(1+C_{5}2\sqrt{\delta_{2,3}}\right)\right)\\
=\exp\left(n^{109/220}\delta_{2,3}+2C_{5}n^{109/220}n^{-1}\right)\le\exp\left(n^{109/220}\delta_{2,3}\right)e^{2C_{5}}.
\end{multline*}
Here we used again (\ref{eq:w_star_dist}) and the definition of $\delta_{2,3}$.

We multiply together all these three last displayed estimates, this
way we can continue our main estimate (\ref{eq:third_term_main_est}).
Note that $\exp\left(n\delta_{2,3}\right)$ cancels, and $\exp\left(-n_{1}\delta_{2,3}\right)$
kills the factor $\exp\left(n^{109/220}\delta_{2,3}\right)$, in more
detail:
\begin{multline*}
\le
\frac{2}{\pi}\int_{\Gamma_{2,3}}\frac{1}{\delta_{2,3}}C_{1}\left(G_{1}^{+}\right)\left(\log n + 1 \right)\left\Vert f\right\Vert _{\partial G_{2}}\exp\left(n\delta_{2,3}\right)e^{2C_{5}}
\\
\cdot\frac{\exp\left(-n\delta_{2,3}-n_{1}\delta_{2,3}-n_{2}\delta_{2,3}\right)}{\delta_{2,3}^{2}}e^{3}\exp\left(n^{109/220}\delta_{2,3}\right)e^{2C_{5}}\left|d\omega\right|
\\
=
\frac{2e^{4C_{5}+3}C_{1}\left(G_{1}^{+}\right)}{\pi}\left\Vert f\right\Vert _{\partial G_{2}}\frac{\log n + 1 }{\delta_{2,3}^{3}}\int_{\Gamma_{2,3}}\left|d\omega\right|
\\
\cdot\exp\left(\left(n^{109/220}-n_{1}\right)\delta_{2,3}\right)\exp\left(-n_{2}\delta_{2,3}\right)
\le
\left\Vert f\right\Vert _{\partial G_{2}}O\left(\frac{n^{2}\log\left(n\right)}{\exp\left(n^{1/12}\right)}\right)
\end{multline*}
where we used several estimates: length of $\Gamma_{2,3}$ is at most
$4\pi$, the definitions of $n_{1},n_{2}$ and $\delta_{2,3}$ and
that $n_{1}>n^{109/220}$, therefore $\exp\left(\left(n^{109/220}-n_{1}\right)\delta_{2,3}\right)\le1$. 

\medskip{}

For $\Gamma_{2,2}$ and $\Gamma_{2,4}$, we apply the same estimate
which we detail for $\Gamma_{2,2}$ only. We again start with the
integral for $w\in\bD$
\begin{multline}
\left|\frac{1}{2\pi i}\int_{\Gamma_{2,2}}\frac{\left(Q\cdot f_{2}\right)\circ\Phi_{1}\circ\psi^{-1}\left(\omega\right)}{w-\omega}\frac{q_{N}\left(w\right)}{q_{N}\left(\omega\right)}d\omega\right|\\
\le\frac{1}{2\pi}\int_{\Gamma_{2,2}}4\frac{1}{\left|w-\omega\right|}\left|\left(Q\cdot f_{2}\right)\left(\Phi_{1}\circ\psi^{-1}\left(\omega\right)\right)\right|\frac{1}{\left|q_{N}\left(\omega\right)\right|}\left|d\omega\right|.\label{eq:second_term_main_est}
\end{multline}
Since $\omega\in\Gamma_{2,2}$, we can rewrite it in the form 
$\omega=\left(1+\delta\right)w_{1}^{*} / \left|w_{1}^{*}\right|$
where $\delta_{2,1}\le\delta\le\delta_{2,3}$ (with $w_{1}^{*}=w_{1}^{*}\left(\delta_{2,3}\right)$
). We use essentially the same steps to estimate $f_{2}$ (the only
one bad guy this time) and $q_{N}$ and $Q$ (this time it is a good
guy). 
In estimating $f_{2}$, the only difference is that $\left|\omega\right|-1=\delta$, so 
\begin{multline*}
\left|f_{2}\left(\Phi_{1}\circ\psi^{-1}\left(\omega\right)\right)\right|
\le
\left\Vert f_{2}\right\Vert _{\partial G_{2}}\exp\left(ng_{G_{2}}\left(\Phi_{1}\circ\psi^{-1}\left(\omega\right),\infty\right)\right)
\\
\le 
C_{1}\left(G_{1}^{+}\right)\left(\log n + 1 \right)\left\Vert f\right\Vert _{\partial G_{2}}\exp\left(n\left(\left|\omega\right|-1\right)\left(1+C_{5}\left|\omega-b_{1}\right|\right)\right)
\\
\le
C_{1}\left(G_{1}^{+}\right)\left(\log n + 1 \right)\left\Vert f\right\Vert _{\partial G_{2}}\exp\left(n\delta+C_{5}n\delta_{2,3}2\sqrt{\delta_{2,3}}\right)
\\
=
C_{1}\left(G_{1}^{+}\right)\left(\log n + 1 \right)\left\Vert f\right\Vert _{\partial G_{2}}\exp\left(n\delta\right)e^{2C_{5}}.
\end{multline*}
Similarly for $q_{N}$, we can write
\begin{multline*}
\frac{1}{\left|q_{N}\left(\omega\right)\right|}\le\frac{1}{\delta_{2,1}^{2}}\frac{1}{\left(1+\delta\right)^{N}}=\frac{1}{\delta_{2,1}^{2}}\exp\left(-\left(n+n_{1}+n_{2}\right)\log\left(1+\delta\right)\right)\\
\le\frac{1}{\delta_{2,1}^{2}}\exp\left(-n\delta-n_{1}\delta-n_{2}\delta+\left(n+n_{1}+n_{2}\right)\frac{\delta_{2,3}^{2}}{2}\right)\\
\le\frac{1}{\delta_{2,1}^{2}}\exp\left(-n\delta-n_{1}\delta-n_{2}\delta\right)\exp\left(3n\ n^{-4/3}\right)\\
\le\frac{\exp\left(-n\delta-n_{1}\delta-n_{2}\delta\right)}{\delta_{2,1}^{2}}e^{3}\le\frac{\exp\left(-n\delta\right)}{\delta_{2,1}^{2}}e^{3}.
\end{multline*}
As for $Q$, we know that $\omega$ is far from $b_{1}$ so $Q$ is
small there. More precisely, following the same argument as for $\Gamma_{2,1}$,
we know that $\sqrt{\delta_{2,3}}/2\ge n_{1}^{-9/10}$, hence (\ref{eq:Q_fast_decreases})
holds for $Q$ at $\omega$, that is, we can write
\[
\left|Q\left(\Phi_{1}\circ\psi^{-1}\left(\omega\right)\right)\right|\le C_{2}\exp\left(-C_{3}n^{1/220}\right).
\]
Putting these all together, we see that $\exp\left(n\delta\right)$
cancels and actually $Q$ make the integrand small. 
So we can continue
the estimate (\ref{eq:second_term_main_est})
\begin{multline*}
\le
\frac{2}{\pi}\int_{\Gamma_{2,2}}\frac{1}{\delta_{2,1}}C_{1}\left(G_{1}^{+}\right)\left(\log n + 1 \right)\left\Vert f\right\Vert _{\partial G_{2}}\exp\left(n\delta\right)e^{2C_{5}}\frac{\exp\left(-n\delta\right)}{\delta_{2,1}^{2}}e^{3}
\\
\cdot C_{2}\exp\left(-C_{3}n^{1/220}\right)\left|d\omega\right|
=
\frac{2e^{2C_{5}+3}C_{2}C_{1}\left(G_{1}^{+}\right)}{\pi}\left\Vert f\right\Vert _{\partial G_{2}}
\int_{\Gamma_{2,2}}\left|d\omega\right|
\\
\cdot\frac{\log n + 1 }{\delta_{2,1}^{3}}\exp\left(-C_{3}n^{1/220}\right)
\le
\left\Vert f\right\Vert _{\partial G_{2}}O\left(\frac{n^{3}\log\left(n\right)}{\exp\left(C_{3}n^{1/220}\right)}\right)
\end{multline*}
where we used that the length of $\Gamma_{2,2}$ is at most $1$ (since
$\delta_{2,3}^{\left(0\right)}<1$) and $\delta_{2,1}=1/n$.

Summarizing these estimates on $\Gamma_{2,1}$, $\Gamma_{2,3}$ and
$\Gamma_{2,2}$ (and also on $\Gamma_{2,4}$), we have uniformly for
$\left|w\right|\le1$, 
\[
\left|p_{2,N}\left(w\right)-\left(Q\cdot f_{2}\right)\circ\Phi_{1}\circ\psi^{-1}\left(w\right)\right|=o\left(1\right)\left\Vert f\right\Vert _{\partial G_{2}}
\]
where $o\left(1\right)$ tends to $0$ as $n\rightarrow\infty$ but
it is independent of $P_{n}$ and $f_{2}$. Obviously, $p_{2,N}\circ\psi$
is a rational function with pole at $v=a_{2}$ only, the order of
the pole at $a_{2}$ (of $p_{2,N}\circ\psi$ ) is $\deg p_{2,N}=N=n+n_{1}+n_{2}=\left(1+o\left(1\right)\right)n$
and using the properties of $w=\psi\left(v\right)$, we uniformly
have for $\left|v\right|\le1$
\[
\left|p_{2,N}\circ\psi\left(v\right)-\left(Q\cdot f_{2}\right)\circ\Phi_{1}\left(v\right)\right|=o\left(1\right)\left\Vert f\right\Vert _{\partial G_{2}},
\]
that is,
\begin{equation}
\left\Vert p_{2,N}\circ\psi-\left(Q\cdot f_{2}\right)\circ\Phi_{1}\right\Vert _{\partial\bD}=o\left(1\right)\left\Vert f\right\Vert _{\partial G_{2}}.\label{eq:sup_norm_err_est_two}
\end{equation}
Since $b_{1}$ is double zero of $q$, $p_{2,N}'\left(b_{1}\right)=\left(\left(Q\cdot f_{2}\right)\circ\Phi_{1}\circ\psi^{-1}\right)'\left(b_{1}\right)$,
and dividing both sides with $\left(\psi^{-1}\right)'\left(b_{1}\right)$,
we obtain 
\begin{equation}
\left(p_{2,N}\circ\psi\right)'\left(1\right)=\left(\left(Q\cdot f_{2}\right)\circ\Phi_{1}\right)'\left(1\right).\label{eq:p_two_n_psi_deriv_at_one}
\end{equation}

\bigskip{}

Consider the ``constructed'' rational function 
\[
h\left(v\right):=\varphi_{1,r}\left(v\right)+p_{1,N}\circ\psi\left(v\right)+p_{2,N}\circ\psi\left(v\right).
\]
This function $h$ has a pole at $a_{1}$ (because of $\varphi_{1,r}$)
and the order of its pole at $a_{1}$ is at most $n$, and $h$ has
a pole at $a_{2}$ (because of $p_{1,N}\circ\psi$ and $p_{2,N}\circ\psi$)
and the order of its pole at $a_{2}$ is at most $N=n\left(1+o\left(1\right)\right)$.
We use the identity 
\[
f\circ\Phi_{1}=\left(Q\cdot f+\left(1-Q\right)\cdot f\right)\circ\Phi_{1}
\]
to calculate the derivatives as follows 
\[
\left(\left(\left(1-Q\right)\cdot f\right)\circ\Phi_{1}\right)'\left(1\right)=\left(\left(1-Q\right)'\cdot f\right)\left(u_{1}\right)\cdot\Phi_{1}'\left(1\right)+\left(\left(1-Q\right)\cdot f'\right)\left(u_{1}\right)\cdot\Phi_{1}'\left(1\right)
\]
where the second term is zero because of the fast decreasing polynomial
($Q\left(u_{1}\right)=1$) and for the first term we can apply Theorem
1.3 from \cite{MR2177185} in the following way ($\left\Vert 1-Q\right\Vert _{\partial G_{2}}\le2$):
\[
\left|\left(1-Q\right)'\left(u_{1}\right)\right|\le\left(1+o\left(1\right)\right)\deg\left(Q\right)2\frac{\partial}{\partial n_{2}\left(u_{1}\right)}g_{G_{2}}\left(u_{1},\infty\right)
\]
where $o\left(1\right)$ depends on $G_{2}$ and $u_{1}$ only and
tends to $0$ as $\deg Q\rightarrow\infty$ (note: $\deg Q\le n^{109/220}\le\sqrt{n}$
). Therefore
\begin{multline}
\left|\left(\left(1-Q\right)'\cdot f\right)\left(u_{1}\right)\cdot\Phi_{1}'\left(1\right)\right|\le\left\Vert f\right\Vert _{\partial G_{2}}\sqrt{n}2\left(1+o\left(1\right)\right)\frac{\partial}{\partial n_{2}\left(u_{1}\right)}g_{G_{2}}\left(u_{1},\infty\right)\\
=\left\Vert f\right\Vert _{\partial G_{2}}O\left(\sqrt{n}\right)\frac{\partial}{\partial n_{2}\left(u_{1}\right)}g_{G_{2}}\left(u_{1},\infty\right)\\
\le o\left(1\right)n\left\Vert f\right\Vert _{\partial G_{2}}\max\left(\frac{\partial}{\partial n_{2}\left(u_{1}\right)}g_{G_{2}}\left(u_{1},\infty\right),\frac{\partial}{\partial n_{1}\left(u_{1}\right)}g_{G_{1}}\left(u_{1},a_{1}\right)\right).
\label{est:one_minus_q_f_prime}
\end{multline}
This way we need to consider
$\left(Q\cdot f\right)\circ\Phi_{1}$ only. The derivatives at $1$
of the original $f$ and $h$ coincide, because of (\ref{eq:phi_decomp}),
(\ref{eq:phi_one_e_deriv_at_one}) and (\ref{eq:p_two_n_psi_deriv_at_one}),
so 
\begin{equation}
h'\left(1\right)=\varphi_{1,r}'\left(1\right)+\left(p_{1,N}\circ\psi\right)'\left(1\right)+\left(p_{2,N}\circ\psi\right)'\left(1\right)=\left(\left(Q\cdot f\right)\circ\Phi_{1}\right)'\left(1\right).\label{eq:h_qf_prime_at_one}
\end{equation}

As for the sup norms, we use (\ref{eq:phi_decomp}), (\ref{eq:phi_one_e_sup_err_est}),
(\ref{eq:sup_norm_err_est_two}), so we write 
\begin{equation}
\left\Vert \left(Q\cdot f\right)\circ\Phi_{1}-h\right\Vert _{\partial\bD}=o\left(1\right)\left\Vert f\right\Vert _{\partial G_{2}}.\label{est:qf_h_sup_norm}
\end{equation}
Now we apply the Borwein-Erdélyi inequality (\ref{ineq:genBorweinErdelyi})
for $h$ as follows:
\begin{equation}
\left|h'\left(1\right)\right|\le\left\Vert h\right\Vert _{\partial\bD}\max\left(\sum_{\alpha}\frac{\partial}{\partial n_{1}\left(1\right)}g_{\bD}\left(1,\alpha\right),\sum_{\alpha}\frac{\partial}{\partial n_{2}\left(1\right)}g_{\bD^{*}}\left(1,\alpha\right)\right)\label{est:main_est_for_h}
\end{equation}
where the summation is taken over all poles in $\bD$ and in $\bD^{*}$
respectively, counting multiplicities. We will continue this estimate
later after simplifying these expressions. Using Propositions \ref{prop:green_deriv_unitdisk}
and \ref{prop:green_deriv_g_one_g_two}, we can write 
\begin{multline*}
\sum_{\alpha}\frac{\partial}{\partial n_{1}\left(1\right)}g_{\bD}\left(1,\alpha\right)\le n\frac{\partial}{\partial n_{1}\left(1\right)}g_{\bD}\left(1,a_{1}\right)=n\frac{\partial}{\partial n_{1}\left(u_{0}\right)}g_{G_{1}}\left(u_{0},F_{1}^{-1}\left[\infty\right]\right)\\
=n\frac{\partial}{\partial n_{2}\left(z_{0}\right)}g_{\bC_{\infty}\setminus K}\left(z_{0},\infty\right)\left|F'\left(u_{0}\right)\right|
\end{multline*}
where in the last step we used Proposition \ref{prop:green_deriv_g_one_g_two}
with $z_{0}=F\left(u_{0}\right)$ and identifying $u_{0}=u_{1}$.
Similarly, we can simplify the second term in the maximum in (\ref{est:main_est_for_h})
\begin{multline*}
\sum_{\alpha}\frac{\partial}{\partial n_{2}\left(1\right)}g_{\bD^{*}}\left(1,\alpha\right)=\deg\left(p_{1,N}+p_{2,N}\right)\frac{\partial}{\partial n_{2}\left(1\right)}g_{\bD^{*}}\left(1,a_{2}\right)\\
\le N\frac{\partial}{\partial n_{2}\left(1\right)}g_{\bD^{*}}\left(1,a_{2}\right)=\left(1+o\left(1\right)\right)n\frac{\partial}{\partial n_{2}\left(u_{0}\right)}g_{G_{2}}\left(u_{0},F_{2}^{-1}\left[\infty\right]\right)\\
=\left(1+o\left(1\right)\right)n\frac{\partial}{\partial n_{1}\left(z_{0}\right)}g_{\bC_{\infty}\setminus K}\left(z_{0},\infty\right)\left|F'\left(u_{0}\right)\right|
\end{multline*}
where $o\left(1\right)$ here does not depend on anything. Note that
we ``used a slightly bit more the pole at $a_{2}$'', but it does
not cause problem. So we can continue the main estimate (\ref{est:main_est_for_h})
\begin{multline*}
\le\left\Vert h\right\Vert _{\partial\bD}\max\Big(n\frac{\partial}{\partial n_{1}\left(u_{0}\right)}g_{G_{1}}\left(u_{0},F_{1}^{-1}\left[\infty\right]\right),\\
\left(1+o\left(1\right)\right)n\frac{\partial}{\partial n_{2}\left(u_{0}\right)}g_{G_{2}}\left(u_{0},F_{2}^{-1}\left[\infty\right]\right)\Big)\\
\le\left\Vert h\right\Vert _{\partial\bD}\left(1+o\left(1\right)\right)n\\
\cdot\max\left(\frac{\partial}{\partial n_{1}\left(u_{0}\right)}g_{G_{1}}\left(u_{0},F_{1}^{-1}\left[\infty\right]\right),\ \frac{\partial}{\partial n_{2}\left(u_{0}\right)}g_{G_{2}}\left(u_{0},F_{2}^{-1}\left[\infty\right]\right)\Big)\right).
\end{multline*}
Summarizing these estimates, we have for $h$
\begin{multline*}
\left|h'\left(1\right)\right|\le\left\Vert h\right\Vert _{\partial\bD}\left(1+o\left(1\right)\right)n\\
\cdot\max\left(\frac{\partial}{\partial n_{1}\left(u_{0}\right)}g_{G_{1}}\left(u_{0},F_{1}^{-1}\left[\infty\right]\right),\ \frac{\partial}{\partial n_{2}\left(u_{0}\right)}g_{G_{2}}\left(u_{0},F_{2}^{-1}\left[\infty\right]\right)\Big)\right).
\end{multline*}
Now we rewrite this inequality for $Q\cdot f$ using (\ref{eq:h_qf_prime_at_one})
and (\ref{est:qf_h_sup_norm}), so
\begin{multline*}
\left|\left(Q\cdot f\right)'\left(u_{1}\right)\right|\le\left\Vert Q\cdot f\right\Vert _{\partial G_{2}}\left(1+o\left(1\right)\right)n\\
\cdot\max\left(\frac{\partial}{\partial n_{1}\left(u_{0}\right)}g_{G_{1}}\left(u_{0},F_{1}^{-1}\left[\infty\right]\right),\ \frac{\partial}{\partial n_{2}\left(u_{0}\right)}g_{G_{2}}\left(u_{0},F_{2}^{-1}\left[\infty\right]\right)\Big)\right)\\
+o\left(1\right)n\left\Vert f\right\Vert _{\partial G_{2}}\cdot\max\left(\frac{\partial}{\partial n_{1}\left(u_{0}\right)}g_{G_{1}}\left(u_{0},F_{1}^{-1}\left[\infty\right]\right),\ \frac{\partial}{\partial n_{2}\left(u_{0}\right)}g_{G_{2}}\left(u_{0},F_{2}^{-1}\left[\infty\right]\right)\Big)\right).
\end{multline*}
Now, we use the estimate $\left\Vert Q\cdot f\right\Vert _{\partial G_{2}}\le\left\Vert f\right\Vert _{\partial G_{2}}$
and (\ref{est:one_minus_q_f_prime}), so 
\begin{multline}
\left|f'\left(u_{1}\right)\right|\le\left\Vert f\right\Vert _{\partial G_{2}}\left(1+o\left(1\right)\right)n\\
\cdot\max\left(\frac{\partial}{\partial n_{1}\left(u_{0}\right)}g_{G_{1}}\left(u_{0},F_{1}^{-1}\left[\infty\right]\right),\ \frac{\partial}{\partial n_{2}\left(u_{0}\right)}g_{G_{2}}\left(u_{0},F_{2}^{-1}\left[\infty\right]\right)\Big)\right).\label{ineq:further_refs}
\end{multline}
In the final step, we use $f=P_{n}\circ F$ and Proposition \ref{prop:Green_trf1},
so we get the main theorem.

\section{Sharpness}

In this section we show that the result is asymptotically sharp, that
is, we prove Theorem \ref{thm:sharpness}. The idea is similar to
that of \cite{MR3019778}. Note that we assume $C^{2}$ smoothness
only.
\begin{proof}
We may assume that 
\[
\frac{\partial}{\partial n_{1}\left(z_{0}\right)}g_{\bC_{\infty}\setminus K}\left(z_{0},\infty\right)\le\frac{\partial}{\partial n_{2}\left(z_{0}\right)}g_{\bC_{\infty}\setminus K}\left(z_{0},\infty\right).
\]
Furthermore, we assume that $n_{1}\left(.\right)$ and $n_{2}\left(.\right)$
are defined on the component of $K$ containing $z_{0}$ and they are continuous there
except for the endpoints. 

It is easy to see that for every $\varepsilon>0$ there exists a compact
set $K^{*}=K^{*}\left(\varepsilon\right)$ such that $\partial K^{*}$
is finite union of disjoint, $C^{2}$ smooth Jordan curves, $K\subset K^{*}$,
$z_{0}\in\partial K^{*}$ and the normal vector $n\left(K^{*},z_{0}\right)$
to $K^{*}$ (pointing outward) at $z_{0}$ is equal to $n_{2}\left(z_{0}\right)$
and
\begin{multline*}
\frac{\partial}{\partial n_{2}\left(z_{0}\right)}g_{\bC_{\infty}\setminus K}\left(z_{0},\infty\right)\left(1-\varepsilon\right)\le\frac{\partial}{\partial n\left(K^{*},z_{0}\right)}g_{\bC_{\infty}\setminus K^{*}}\left(z_{0},\infty\right)\\
\le\frac{\partial}{\partial n_{2}\left(z_{0}\right)}g_{\bC_{\infty}\setminus K}\left(z_{0},\infty\right).
\end{multline*}
These conditions, roughly speaking, require that near $z_{0}$, $K^{*}$
is on the $n_{1}\left(z_{0}\right)$-side of $K$ and the whole $K^{*}$
shrinks to $K$ as $\varepsilon\rightarrow0$. 
Figure \ref{fig:kstar}
depicts $K$ and the grey area is $K^{*}$.

\begin{figure}
\begin{centering}
\includegraphics[width=1\textwidth]{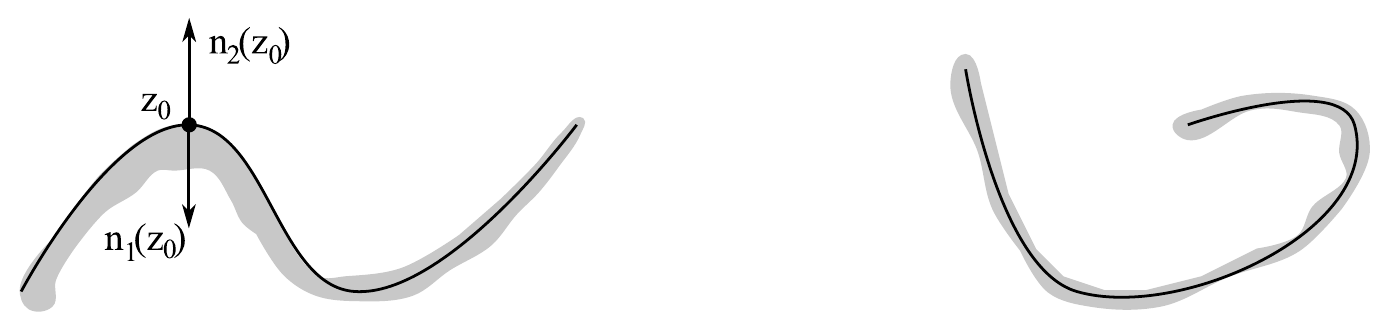}
\par\end{centering}

\protect\caption{The sets $K$ and $K^{*}$}

\label{fig:kstar}
\end{figure}

Now we apply the sharpness result of \cite{MR2177185} (Theorem 1.4,
p. 194). This gives a sequence of polynomials for $K^{*}\left(\varepsilon\right)$,
say $P_{\varepsilon,n}$, with $\deg P_{\varepsilon,n}\le n$ such
that 
\begin{multline*}
\left|P_{\varepsilon,n}'\left(z_{0}\right)\right|\ge n\left(1-o_{\varepsilon}\left(1\right)\right)\left\Vert P_{\varepsilon,n}\right\Vert _{K^{*}\left(\varepsilon\right)}\frac{\partial}{\partial n\left(K^{*},z_{0}\right)}g_{\bC_{\infty}\setminus K^{*}}\left(z_{0},\infty\right)\\
\ge n\left(1-o_{\varepsilon}\left(1\right)\right)\left(1-\varepsilon\right)\left\Vert P_{\varepsilon,n}\right\Vert _{K}\frac{\partial}{\partial n_{2}\left(z_{0}\right)}g_{\bC_{\infty}\setminus K}\left(z_{0},\infty\right)
\end{multline*}
where $o_{\varepsilon}\left(1\right)$ depends on $K^{*}\left(\varepsilon\right)$
and $z_{0}$ and tends to $0$ as $\deg P_{\varepsilon,n}\rightarrow\infty$.
Since $\varepsilon$ was arbitrary, we see that $\left(1-o_{\varepsilon}\left(1\right)\right)\left(1-\varepsilon\right)=1-o\left(1\right)$,
that is, choosing a suitable subsequence of $\left\{ P_{\varepsilon,n}\right\} $
we obtain the assertion.
\end{proof}

\section*{Acknowledgement}

The first author was supported by the European Research Council Advanced
grant No. 267055, while he had a postdoctoral position at the Bolyai
Institute, University of Szeged, Aradi v. tere 1, Szeged 6720, Hungary.

The second author, Béla Nagy was supported by Magyary scholarship:
This research was realized in the frames of TÁMOP 4.2.4. A/2-11-1-2012-0001
,,National Excellence Program - Elaborating and operating an inland
student and researcher personal support system.'' The project was
subsidized by the European Union and co-financed by the European Social
Fund. 

The authors are deeply indebted to Professor Vilmos Totik for several
helpful suggestions and comments which enabled us to improve the presentation
of our result and for drawing our attention to Widom's open-up result.

\section*{References}

%\bibliographystyle{amsalpha}
%\bibliography{amshiv2}

\begin{thebibliography}{AKML98}

\bibitem[AKML98]{MR1639759}
Dmitri Alekseevsky, Andreas Kriegl, Peter~W. Michor, and Mark Losik,
  \emph{Choosing roots of polynomials smoothly}, Israel J. Math. \textbf{105}
  (1998), 203--233. \MR{1639759 (2000c:58017)}

\bibitem[BE95]{MR1367960}
Peter Borwein and Tam{\'a}s Erd{\'e}lyi, \emph{Polynomials and polynomial
  inequalities}, Graduate Texts in Mathematics, vol. 161, Springer-Verlag, New
  York, 1995. \MR{1367960 (97e:41001)}

\bibitem[BE96]{MR1433285}
\bysame, \emph{Sharp extensions of {B}ernstein's inequality to rational
  spaces}, Mathematika \textbf{43} (1996), no.~2, 413--423 (1997). \MR{1433285
  (97k:26014)}

\bibitem[Con95]{MR1344449}
John~B. Conway, \emph{Functions of one complex variable. {II}}, Graduate Texts
  in Mathematics, vol. 159, Springer-Verlag, New York, 1995. \MR{1344449
  (96i:30001)}

\bibitem[DK07]{MR2380804}
V.~N. Dubinin and S.~I. Kalmykov, \emph{A majorization principle for
  meromorphic functions}, Mat. Sb. \textbf{198} (2007), no.~12, 37--46.
  \MR{2380804 (2009c:30074)}

\bibitem[GG76]{MR0417417}
A.~A. Gon{\v{c}}ar and L.~D. Grigorjan, \emph{Estimations of the norm of the
  holomorphic component of a meromorphic function}, Mat. Sb. (N.S.) \textbf{28}
  (1976), no.~4, 571--575.

\bibitem[Kal08]{MR2765937}
S.~I. Kalmykov, \emph{Majorization principles and some inequalities for
  polynomials and rational functions with prescribed poles}, Zap. Nauchn. Sem.
  S.-Peterburg. Otdel. Mat. Inst. Steklov. (POMI) \textbf{357} (2008),
  no.~Analiticheskaya Teoriya Chisel i Teoriya Funktsii. 23, 143--157, 227.
  \MR{2765937 (2012b:30008)}

\bibitem[LMR95]{MR1332889}
Xin Li, R.~N. Mohapatra, and R.~S. Rodriguez, \emph{Bernstein-type inequalities
  for rational functions with prescribed poles}, J. London Math. Soc. (2)
  \textbf{51} (1995), no.~3, 523--531. \MR{1332889 (96b:30005)}

\bibitem[Nag05]{MR2105685}
B{\'e}la Nagy, \emph{Asymptotic {B}ernstein inequality on lemniscates}, J.
  Math. Anal. Appl. \textbf{301} (2005), no.~2, 449--456. \MR{2105685
  (2006c:41019)}

\bibitem[NT05]{MR2177185}
B{\'e}la Nagy and Vilmos Totik, \emph{Sharpening of {H}ilbert's lemniscate
  theorem}, J. Anal. Math. \textbf{96} (2005), 191--223. \MR{2177185
  (2006g:30008)}

\bibitem[NT13]{MR3019778}
\bysame, \emph{Bernstein's inequality for algebraic polynomials on circular
  arcs}, Constr. Approx. \textbf{37} (2013), no.~2, 223--232. \MR{3019778}

\bibitem[Pom92]{MR1217706}
Ch. Pommerenke, \emph{Boundary behaviour of conformal maps}, Grundlehren der
  Mathematischen Wissenschaften [Fundamental Principles of Mathematical
  Sciences], vol. 299, Springer-Verlag, Berlin, 1992. \MR{1217706 (95b:30008)}

\bibitem[Ran95]{MR1334766}
Thomas Ransford, \emph{Potential theory in the complex plane}, London
  Mathematical Society Student Texts, vol.~28, Cambridge University Press,
  Cambridge, 1995. \MR{1334766 (96e:31001)}

\bibitem[SFS89]{MR1007599}
Yu.~V. Sidorov, M.~V. Fedoryuk, and M.~I. Shabunin, \emph{Lektsii po teorii
  funktsii kompleksnogo peremennogo}, third ed., ``Nauka'', Moscow, 1989.
  \MR{1007599 (90e:30001)}

\bibitem[SL68]{MR0229803}
V.~I. Smirnov and N.~A. Lebedev, \emph{Functions of a complex variable:
  {C}onstructive theory}, Translated from the Russian by Scripta Technica Ltd,
  The M.I.T. Press, Cambridge, Mass., 1968. \MR{0229803 (37 \#5369)}

\bibitem[ST97]{MR1485778}
Edward~B. Saff and Vilmos Totik, \emph{Logarithmic potentials with external
  fields}, Grundlehren der Mathematischen Wissenschaften [Fundamental
  Principles of Mathematical Sciences], vol. 316, Springer-Verlag, Berlin,
  1997, Appendix B by Thomas Bloom. \MR{1485778 (99h:31001)}

\bibitem[Sto62]{Stoilow}
S.~Stoilov, \emph{Teoria funktsii kompleksnogo peremennogo, vol. {II}}, Inost.
  Lit., Moscow, 1962.

\bibitem[Tot10]{MR2574887}
Vilmos Totik, \emph{Christoffel functions on curves and domains}, Trans. Amer.
  Math. Soc. \textbf{362} (2010), no.~4, 2053--2087. \MR{2574887 (2011b:30006)}

\bibitem[Wid69]{MR0239059}
Harold Widom, \emph{Extremal polynomials associated with a system of curves in
  the complex plane}, Advances in Math. \textbf{3} (1969), 127--232 (1969).
  \MR{0239059 (39 \#418)}

\end{thebibliography}

\providecommand{\bysame}{\leavevmode\hbox to3em{\hrulefill}\thinspace}
\providecommand{\MR}{\relax\ifhmode\unskip\space\fi MR }
\providecommand{\MRhref}[2]{%
  \href{http://www.ams.org/mathscinet-getitem?mr=#1}{#2}
}
\providecommand{\href}[2]{#2}

\bigskip

Sergei Kalmykov
\\
Institute of Applied Mathematics, FEBRAS, 7 Radio Street, Vladivostok,
690041, Russia,
\\
Far Eastern Federal University, 8 Sukhanova Street, Vladivostok, 690950,
Russia and
\\
Bolyai Institute, University of Szeged, Szeged, Aradi v. tere 1, 6720,
Hungary
\\
email address: \href{mailto:sergeykalmykov@inbox.ru}{sergeykalmykov@inbox.ru}

\medskip{}

Béla Nagy
\\
MTA-SZTE Analysis and Stochastics Research Group, Bolyai Institute,
University of Szeged, Szeged, Aradi v. tere 1, 6720, Hungary
\\
email address: \href{mailto:nbela@math.u-szeged.hu}{nbela@math.u-szeged.hu}

\end{document}